\documentclass[11pt,reqno]{amsart}
\usepackage{fullpage,amsfonts,amsmath,amssymb}

\usepackage{amsfonts,amsmath,amssymb,fullpage,graphicx}
\usepackage{verbatim}

\newcommand{\R}{\mathbb R}
\newcommand{\C}{\mathbb C}
\newcommand{\Z}{\mathbb Z}%
\newcommand{\N}{\mathbb N}
\newcommand{\what}{\widehat}
\newcommand{\setm}{\setminus}
\newtheorem{theorem}{Theorem}[section]
\newtheorem{lemma}[theorem]{Lemma}

\theoremstyle{definition}

\theoremstyle{definition}
\newtheorem{remark}[theorem]{Remark}

\numberwithin{equation}{section}

\begin{document}
\title{A Genuine analogue of WIENER TAUBERIAN THEOREM FOR  $ \mathrm {SL}(2, \R)$}
\author{Tapendu Rana}

\maketitle
\begin{abstract}
We prove a genuine analogue of Wiener Tauberian theorem for integrable functions on  $\mathrm {SL}(2, \R).$  
\end{abstract}
 \section{\textbf{Introduction}}
 
Let $f \in L^1(\R)$ and $\what{f}$ be its Fourier transform. The celebrated   Wiener-Tauberian (W-T) theorem says that the ideal generated by $f$ in $L^1(\R)$ is dense in $L^1(\R)$ if and only if $\what{f}$ is  nowhere vanishing on $\R$. This theorem has been extended to abelian groups. In 1955, Ehrenpreis and Mautner observed that the exact analogue of the  theorem above fails for the commutative algebra of the integrable $K$ -biinvariant functions on the group $\mathrm {SL}(2, \R) $, where $K = \mathrm{SO}(2)$ is a maximal compact subgroup. Nonetheless the authors  proved that if a $K $-biinvariant integrable function $f$ on $G$ satisfies a ``not-to-rapid decay" condition and  nonvanishing condition on a extended strip $S_{1,\delta} = \{ \lambda \in \C \mid |\Re \lambda| \leq 1+ \delta \} $ for $\delta >0$, etc, that is,
\begin{equation*} 
\what{f}(\lambda)\not=0 \text{ for all } \lambda\in S_{1, \delta},
\end{equation*}
and ``not-to-rapid decay" condition
\begin{equation*}
\limsup_{|t|\rightarrow\infty}|\what{f}(it)|e^{Ke^{|t|}}>0 \text{ for all  } K>0 
\end{equation*}
  together with some other conditions then the ideal generated by $f$ in $L^1(G//K)$ is dense in $L^1(G//K)$ (see \cite{EM} for the precise statements).  Using the extended strip condition the results has been generalised to the full group $\mathrm {SL}(2, \R)$ (see \cite{Sarkar-1997}) and to the real rank one semi simple Lie groups (see \cite{Ben}, \cite{Ben-1}, \cite{Sarkar-1998}, \cite{Sitaram}). We also refer \cite{Naru-2009}, and \cite{Naru-2011}  for an analogue of W-T theorem for semisimple Lie groups of arbitrary real rank.
  
  Y. Ben Natan, Y. Benyamini, H. Hedenmalm and Y. Weit (in \cite{ Ben,Ben-1}) proved a  genuine analogue  of  the W-T theorem without the extended strip condition for $L^1(\mathrm {SL}(2, \R)//\mathrm{SO}(2))$. In \cite{PS} the authors extended this result to real rank one semisimple Lie group in the $K$ -biinvariant setting.  In this article we generalize the result to the full group $\mathrm {SL}(2, \R)$ and therefore this improves the corresponding result of \cite{Sarkar-1997}..
  
  Let $G$ be the group $ \mathrm {SL}(2, \R)$ and $K$ be its maximal compact subgroup $\mathrm{SO}(2)$. A complex valued function $f$ on $G$ is said to be of left (resp. right) $K$-type $n$ if \begin{equation}
 f(kx) =e_n(k) f(x)\text{ (resp. } f(xk) =e_n(k) f(x) ) \text{ for all } k \in K \text{ and } x \in G,
  \end{equation} where $e_n(k_\theta) = e^{in \theta}$. For a class of functions $\mathcal{F} $ on $G$ (e.g. $L^1(G)$),  $\mathcal{F}_n $ denotes the corresponding subclass of functions of right $n$ type and  $\mathcal{F}_{m,n} $ will denote the subclass of $\mathcal{F}_n$ which are also of left type $m$. We denote the subclass of  $\mathcal{F} $ consisting of functions with integral zero  by $\mathcal{F}^0 $.
   
   The main result (Theorem \ref{WTT for L1(G)}) of this article is an analogue of W-T theorem to the the full group without the redundant extended strip condition. We first prove the W-T theorem for $L^1(G)_{n,n}$ (Theorem \ref{WTT for L1(G)n,n}) for all $n \in \Z$. This is the most crucial step in the direction of proving the W-T theorem to the full group. %Next we will extend the W-T theorem for $L^1(G)_{n,n}$  to $L^1(G)$ using the same technique in \cite{Sarkar-1997}.
   Before stating our main result we introduce some notation. For a function $f \in L^1(G)$ its principal and discrete parts of the Fourier transform will be denoted by $\what{f}_H$  and $\what {f}_B$ respectively. Let $M =\{ \pm I\}$ and  $\what M = \{ \sigma^+, \sigma^- \}$ which consists  the trivial ($\sigma^+$) and the non-trivial ($\sigma^-$) irreducible representations of $M$. The representation $\pi_{\sigma^-, 0}$  has two irreducible subrepresentations, so called mock series. We will denote them  by $D_+$ and $D_-$. The representation spaces of $D_+$ and $D_-$ contain  $e_n \in L^2(K)$ respectively for positive odd $n$'s and negative odd $n$'s. For each $\sigma \in \what M$, $\Z^{\sigma} $ stands for the set of even integers for $\sigma =\sigma^+$, and the set of odd integers for $\sigma =\sigma^-$.  Moreover, we express $-\sigma$ by $-\sigma^+ =\sigma^-$ and $-\sigma^- = \sigma^+$. We define,
\begin{align*}
& S_1 = \lbrace \lambda \in \C \mid |\Re \lambda | \leq 1 \rbrace \text{ and }
& \Gamma_n =\begin{cases} \lbrace
 k \mid 0 <k < n \text { and } k \in \Z^{-\sigma} \rbrace \text{ if } n > 0 \\
\lbrace k \mid n <k < 0\text { and } k \in \Z^{-\sigma} \rbrace \text{ if } n < 0\end{cases}.
\end{align*}
  For any function $F$ on $i\R$, we let 
 \begin{equation}
 \delta_\infty^{\pm}(F)=-\limsup_{t\rightarrow\infty} e^{-\frac{\pi}{2}t}\log|F(\pm it)|. 
 \end{equation}
 We now state our main theorem.
 \begin{theorem}\label{WTT for L1(G)}
 Let $\{f^\alpha\mid \alpha\in \Lambda\}$ be a collection of functions in $L^1(G)$  such that the collections $\{\what{f^\alpha_H} \mid\alpha\in \Lambda\}$  and $\{\what{f^\alpha_B} \mid\alpha\in \Lambda\}$  have  no common zero in $\what{M} \times S_1 \cup \lbrace D_+ , D_- \rbrace$ and $Z^*$ respectively. If  $\inf\limits_{\alpha\in\Lambda \, m,n  \in \Z}\delta_\infty^{\pm}(\what{f^\alpha_H})_{m,n}=0  $, then the $ L^1(G) $ -bimodule generated by  $\{f^\alpha\mid \alpha\in \Lambda\}$ is dense in $ L^1(G). $
 
 Moreover, if the integral of $f^\alpha$ is zero for all $\alpha$, then the ideal is dense in $L^1(G)^0$.
\end{theorem}

For $f \in L^1(G)_n$, the natural domain of the principal part $\what f_H$ and the discrete part  $\what f_B$ of the Fourier transform is $S_1$ and $\Gamma_n$ respectively. To prove Theorem \ref{WTT for L1(G)} we will prove the following theorem.

\begin{theorem}\label{WTT for L1(G)n}	
 	Let $\{f^\alpha\mid \alpha\in \Lambda\}$ be a collections of functions in $L^1(G)_{n}$  such that the collection $\{\what{f^\alpha_H} \mid\alpha\in \Lambda\}$  and $\{\what{f^\alpha_B} \mid\alpha\in \Lambda\}$  have  no common zero in $S_1$ and $\Gamma_n$ respectively . Moreover, if  $\inf\limits_{\alpha\in\Lambda, m \in \Z}\delta_\infty^{\pm}(\what{f^\alpha_H})_{m,n}=0 $, then the left $ L^1(G) $ module generated by  $\{f^\alpha\mid \alpha\in \Lambda\}$ is dense in $ L^1(G)_n. $
 \end{theorem}

 Theorem \ref{WTT for L1(G)n} will follow from the theorem below.
 \begin{theorem}\label{WTT for L1(G)n,n}
 Let $\{f^\alpha\mid \alpha\in \Lambda\}$ be a collection of functions in $L^1(G)_{n,n}$ and $I$ be the smallest closed ideal in $L^1(G)_{n,n}$ containing $\{f^\alpha\mid \alpha\in \Lambda\}$ such that the collection $\{(\what{f^\alpha_H})_{n,n} \mid\alpha\in \Lambda\}$  and $\{(\what{f^\alpha_B})_{n,n} \mid\alpha\in \Lambda\}$  have  no common zero in $S_1$ and $\Gamma_n$ respectively . Moreover, if  $\inf\limits_{\alpha\in\Lambda}\delta_\infty^{\pm}(\what{f^\alpha_H})_{n,n}=0$, then $I= L^1(G)_{n,n}.$
 \end{theorem}
 Proof of the theorem above borrows heavily from the ideas and methods of \cite{Ben}, \cite{PS}  which uses the method of the resolvent transform.  In the following we give a sketch of our proof.
 \begin{enumerate}
\item[1.] We will begin by showing that for all $\lambda $ in $\C^+ = \{ \lambda \in \C \mid \Re \lambda > 0 \}$ except for a finite set $ \textbf{B}$  there is a family $b_{\lambda}$ such that $\what{b_{\lambda}}_H (i\xi) = \frac{1}{\lambda^2+ \xi^2}$ for all $\xi \in \R$ and $\what{b_{\lambda}}_B (k) = \frac{1}{\lambda^2 -k^2}$ for all $k \in \Gamma_n$. For $\Re \lambda >1$, $b_{\lambda} \in \L^1(G)_{n,n}$ and    $\{ b_{\lambda} \mid \Re \lambda >1 \text{ and } \lambda \notin \textbf{B} \}$ spans a dense subset of $L^1(G)_{n,n}$. We will show  $||b_{\lambda}||_1 \rightarrow 0$ if $\lambda \rightarrow \infty $ along the positive real axis.
\item[2.] By the Banach algebra theory (using the fact that principal part and discrete part of the Fourier transforms of the  elements of $I$ have no common zero), we define $\lambda\mapsto B_\lambda$   as a $L^1(G)_{n,n}/I$ valued even entire function.
 \item[3.] Let $g\in L^\infty(G)_{n,n}$ such that $g$ annihilates $I$. We define the resolvent transfrom $\mathcal R[g]$ by $$\mathcal R[g](\lambda)=\langle B_\lambda, g\rangle, $$ Considering $g$ as a bounded linear functional on $L^1(G)_{n,n}/I$, we write $$\mathcal R[g](\lambda)=\langle B_\lambda, g\rangle, $$ where $B_\lambda=b_\lambda + I\in L^1(G)_{n,n}/I$, for all $\lambda$ with $\Re\lambda>1$ and $\lambda \notin \textbf{B}$.
 \item[4.] We need an explicit formula of the function $\mathcal{R}[g](\lambda )$.  For this we find a representative  $T_\lambda f$ in $L^1(G)_{n,n}$ of the cosets $B_\lambda $  for $0 <\Re \lambda < 1$  where $f \in I$  such that $(\what{f}_H(\sigma,\lambda))_{n,n} \neq 0$. %So we will get an explicit expression of $\mathcal R[g](\lambda)$ for  $0 <\Re \lambda < 1$.
 \item[5.] By the estimates of $||b_\lambda||_1$, $||T_\lambda f||_1$ and using a continuity argument we get the necessary estimate of $\mathcal{R}[g](\lambda)$. Then using a log-log type theorem \cite[Theorem 6.3]{PS} we show $\mathcal{R}[g] = 0 .$ 
 \item[6.] By denseness  of  $\{ b_\lambda \mid \Re \lambda >1 \text{ and } \lambda \notin \textbf{B} \}$ , it follows that $g =0$.
 \end{enumerate}
 
 Let $\Omega$ be the Casimir element of $G$. In \cite{PS} the solutions $\phi_{\sigma,\lambda}^{0,0}$ and $\Phi_{\sigma,\lambda}^{0,0}$ of 
 \begin{equation}
\Omega f = \left( \frac{\lambda ^2 -1}{4}\right)f \label{Omega equation}
 \end{equation} played a crucial role and they are given in terms of hypergeometric functions in  \cite{Erdelyi-1}.
 %In our case we could not find the formula for second solution $\Phi_{\sigma,\lambda}^{n,n}$ of the equation above, which plays one of the most fundamental part in our proof. 
  From \cite[p.31 ]{Special Fns} observing a formula of $\phi_{\sigma,\lambda}^{n,n}$ we found a way to derive the second solution of \eqref{Omega equation} in terms of hypergeometric functions. We are also able to write $\phi_{\sigma,\lambda}^{n,n}$ as a linear combination of $\Phi_{\sigma,\lambda}^{n,n}$ and $\Phi_{\sigma,-\lambda}^{n,n}$ that is,
 \begin{equation}
  {\phi_{\sigma,\lambda}^{n,n}} = c^{n,n}_\sigma (\lambda) \Phi_{\sigma,\lambda}^{n,n} + c^{n,n}_\sigma (-\lambda){\Phi_{-\lambda}^{n,n}}.
  \end{equation} It is an essential tool to find the Fourier transforms of $b_\lambda$'s.
 
 As in \cite[Lemma 8.1]{PS}, using asymptotic behaviour $\Phi_{\sigma,\lambda}^{n,n}(a_t)$ for $\lambda \in \C$ near $t=\infty$  we show,
 \begin{equation}
 \lim_{t\rightarrow\infty}\dfrac{\dfrac{\phi_{\sigma,i\xi}^{n,n}}{\Phi_{\sigma,\lambda}^{n,n}}(a_t)}{e^{2\lambda t}}=0. \label{limit of phi xi by Phi lambda}
 \end{equation} This directly gives $\what{b_{\lambda}}_H(i\xi)= \frac{1}{\lambda^2 +\xi^2}$, for all $\xi \in \R$. But for a general $n $, we also need to find  $\what{b_{\lambda }}_B(k)$  for all $k \in \Gamma_n$ and  for that we need to show 
 \begin{equation}
 \lim_{t\rightarrow\infty}\dfrac{\dfrac{\phi_{\sigma,|k|}^{n,n}}{\Phi_{\sigma,\lambda}^{n,n}}(a_t)}{e^{2\lambda t}}=0. \label{limit of phi k by Phi lambda} 
 \end{equation}
 Here asymptotic behaviour of  $\Phi_{\sigma,\lambda}^{n,n}$ is not enough.
 We have to use the full potential of decay of discrete series matrix coefficient $\phi_{\sigma,|k|}^{n,n}$. From \cite[Theorem 8.1]{Barker} we get $\phi_{\sigma,|k|}^{n,n}$ has sufficient decay as $t \rightarrow \infty$. Using this we proved \eqref{limit of phi k by Phi lambda} and consequently $\what{b_{\lambda B}}(k) =\frac{1}{\lambda^2 -k^2}$  for all $k \in \Gamma_n$.
 
By inverse Fourier transform we find the representative of $B_\lambda$. For all but finitely many $\lambda$ with $\Re \lambda > 1$ we show $B_\lambda =  b_\lambda +I $. In general there are some zeros of $c^{n,n}_\sigma(-\lambda)$ in $\C^+$ for $n\in \Z$. For this reason we have to remove a neighbourhood $\textbf{B}_1$ to find estimate  of  $||b_\lambda||_1$   on $|\Re \lambda | >1$. Using estimate of $||b_\lambda||_1$ we find the estimate  of $\mathcal{R}[g](\lambda)$ on $\{ \lambda \in \C \mid |\Re \lambda | >1 \} \setm \textbf{B}_1$. We use continuity to get estimate of $\mathcal{R}[g](\lambda)$ on $\textbf{B}_1$. Similarly using the estimate of $||T_\lambda f||_1$ we find the estimate of  $\mathcal{R}[g](\lambda)$ on $0 < \Re \lambda <1$ and cosequently we find the necessary estimate of $\mathcal{R}[g](\lambda)$. Finally using a log-log type theorem it will follow that   $\mathcal{R}[g] = 0$.
 
  Next we extend W-T theorem for $L^1(G)_{n,n}$ to $L^1(G)_n$. %using the technique in \cite .
  From the given collection $\lbrace f^\alpha \mid \alpha \in \Lambda \rbrace \subset L^1(G)_{n} $ and using the isomorphism between $L^1$-Schwartz space and its image under Fourier transform  we construct a new collection of $L^1(G)_{n,n}$ functions $\lbrace  g_m * f^\alpha \mid m \in \Z^\sigma \rbrace$. We show the new collection satisfy the hypothesis of Theorem \ref{WTT for L1(G)n,n}. Therefore the collection $\lbrace  g_m * f^\alpha \mid m \in \Z^\sigma \rbrace$ is dense in $L^1(G)_{n,n}$ and consequently Theorem \ref{WTT for L1(G)n} will follow. Following similar idea as above we prove W-T theorem for $L^1(G)$.\\

\section{\textbf{Preliminaries}}
 In this article most of our notations are standard can be found in  \cite{Barker}, \cite{Special Fns}, \cite{PS}  and \cite{Sarkar-1997}. We will denote $C$ as constant and its value can change from one line to another. For any two positive expressions $f_1$ and $f_2$, $f_1 \asymp f_2 $ stands for that there are positive constants $C_1$, $C_2$ such that $C_1 f_1 \leq f_2 \leq C_2 f_1 $. For $ z \in \C$ we will use $\Re z$ and $\Im z$ to denote real and imaginary parts of $z$ respectively. 
 
 For $k \in \Z^*$ and $\sigma \in \what M$ be determined by $k \in \Z^{-\sigma}$ and we define, \begin{equation}
 \Z(k) = \begin{cases}\{ m \in \Z^{\sigma} : m \geq k+1 \} \text{ if }  k\geq 1 \\
 \{ m \in \Z^{\sigma} : m \leq k-1 \} \text { if }  k\leq -1.
 \end{cases}
 \end{equation}
 
 The  Iwasawa decomposition for G gives a gives a diffeomorphism of $K \times A \times N$ onto $G$ where $A=\{ a_t \mid t\in \R \}$ and $N=\{ n_\xi \mid \xi \in \R \}$.
%The group $G= \mathrm {SL}(2, \R)$ has the Iwasawa decomposition $G=KAN$ and the Cartan decomposition $G={K \overline{A_+}K}$. %should I explain A K and N
  That is, by Iwasawa decomposition  $x \in G$ can be uniquely written as $x =k_\theta  a_t n_\xi $ and using this we define $K(x) = k_\theta$ and $H(x) = t$. Let $A^+=\{ a_t= \mid t>0 \}$. The Cartan decomposition for $G$  gives $G={K \overline{A_+}K}$. Let $dg$, $dn$, $dk$ and $dm$ be the Haar meausres of $G$, $N$, $K$ and $M$  respectively where $\int_K dk=1$ and $\int_M dm = 1$.  We have the following integral formulae corresponding to the Cartan decomposition,
which holds for any integrable function:
\begin{equation}
 \int_G f(x) dx = \int_K \int_{\R^+} \int_K f(k_1 a_t k_2) \Delta(t) dk_1 dt dk_2,
\end{equation}
where $\Delta(t)= 2 \sinh 2t$. \\

%A function is called $(n,n)$ type if $$f(k_1 x k_2) = e_n(k_1) f(x)  e_n(k_2)$$ for all $x$ $\in $ $G, k_1, k_2$ $\in K$, where $e_n(k_\theta) = e^{in \theta}$.
%We denote the set of all $(n,n)$ type functions in $L^1 (G)$ by ${L^1 (G)}_{n,n}$.

For all $\lambda \in \C $  let us define,
$${\phi_{\sigma,\lambda}^{n,n}}(x) =\int_K e^{(\lambda-1)H(xk)} e_n(k^{-1}) \overline{e_n(K(xk)^{-1})} dk, \quad \text{ for all } x \in G .$$
Then we have for all $\lambda \in \C$, $ {\phi_{\sigma,\lambda}^{n,n}}$ is a smooth eigenfunctions of the Casimir element  $\Omega $ that is,
$$ \Omega {\phi_{\sigma,\lambda}^{n,n}} = \frac{{\lambda}^2-1}{4}  {\phi_{\sigma,\lambda}^{n,n}}.$$
Let $\Pi_{n,n}(\Omega)$ be the differential operator on $A\setminus \lbrace1\rbrace$ defined by 

$$\Pi_{n,n}(\Omega) f = \frac{d^2 }{dt^2} f(a_t) + 2 \coth 2t \frac{d}{dt}   f(a_t) + \frac{1}{4}  \frac{n^2}{\cosh^2 t} f(a_t), t>0.$$ 
Then from  \cite[p.62, eqn. (13.2)]{Barker}) we get that ${\phi_{\sigma,\lambda}^{n,n}}$ is a solution of the following equation,
%For simplicity if $f \in L^1(G)_{n,n}$ then we denote the principal and discrete parts of the Fourier transform by $\what{f_H} (\lambda)$ and $\what{f_B}(k)$ respectively. %(instead of $(\what{f_H}(\lambda))_{n,n}$  and $(\what{f_B}(k))_{n,n}$ ).
\begin{equation} \label{Equation of eigen function of Pi {n,n}}
 \Pi_{n,n}(\Omega) f = (\lambda^2-1)  f .
\end{equation}

We also have the following properties of ${\phi_{\sigma,\lambda}^{n,n}}$:
 \begin{enumerate}
\item ${\phi_{\sigma,\lambda}^{n,n}}$ is a $(n,n)$ type function. 
\item ${\phi_{\sigma,\lambda}^{n,n}} = {\phi_{\sigma,-\lambda}^{n,n}}$, ${\phi_{\sigma,\lambda}^{n,n}}(a_t)={\phi_{\sigma,\lambda}^{n,n}}(a_{-t)}$.
\item For any fixed $x \in G$, $\lambda \mapsto {\phi_{\sigma,\lambda}^{n,n}}(x)$ is an entire function.
\item $|{\phi_{\sigma,\lambda}^{n,n}}(x)| \leq 1$ $x \in G$ if $\lambda \in S_1.$ 
\end{enumerate}

For $f \in  L^1(G)_{n,n}$ the principal and discrete parts of the Fourier transform are defined by,
 \begin{align}
    \what{f_H}(\sigma,\lambda)_{n,n}&= \int_G f(x) {\phi_{\sigma,\lambda}^{n,n}}(x^{-1}) dx \quad \text{ for all } \lambda \in S_1 ,\\ 
   \what{f_B}(k)_{n,n} &= \int_G f(x) { \psi_{k}^{n,n}}(x^{-1}) dx \quad \text{ for all } k \in \Gamma_n . 
 \end{align}
It follows from Riemann-Lebesgue lemma that if $f \in L^1(G)_{n,n}$ then $ |\what{f_H}(\sigma,\lambda)_{n,n}| \rightarrow  0$
as $|\Im \lambda| \rightarrow \infty  $ in $S_1$. We also have from \cite[p.30 propn 7.3]{Barker} 
\begin{equation}
{\psi_{k}^{n,n}} ={\phi_{\sigma,|k|}^{n,n}} \text{ for all $k \in \Gamma_n$ }.
\end{equation}
We denote $\mathcal{C}^1(G)$ the $L^1$- Schwartz space of $G$ . Suppose $\sigma \in \what M, m,n \in \Z^\sigma$ then the space $\mathcal{C}^1_H(\what G)_{m,n}$ denotes the collection of functions $F : S_1 \rightarrow \C$ such that
\begin{itemize}
\item[(1)] $F$ is continuous on $S_1$ and homomorphic on Int $S_1$,
\item[(2)] $F(\lambda) = \varphi^{m,n}_\lambda F(-\lambda)$ for all $\lambda \in S_1$, where  \begin{equation}\label{peref definition of Pmn lambda}
\varphi^{m,n}_\lambda  = P_{m,n}(\lambda)/P_{m,n}(-\lambda)
 \end{equation} is the rational function defined in \cite[Prop. 7.1]{Barker}, 
\item[(3)] $\what{\rho}_{H,l,r} (F) < \infty$ for all $ l \in \N$, $r\in \R^+,$ where $$\what{\rho}_{H,l,r} (F) = \sup_{\lambda \in S_1 } \left| \left( \dfrac{d}{d\lambda} \right)^l F(\lambda) \right| (1+|\lambda|)^r,$$ 
\item[(4)] $F(k) = 0 $ if $nm < 0 $, $k \in \Z^{-\sigma}$ and $|k| \leq \min \{ |m|,|n|,1 \}.$ 
\end{itemize}

\noindent We note that for the particular case $m = n $, $P_{n,n}=1$  so the property 2 in the definition of  $\mathcal{C}^1_H(\what G)_{n,n}$ reduces to $F(\lambda) = F(-\lambda) $ and property 4 becomes irrelevant. Let,
\begin{align*}
\Z^1_{m,n} =\begin{cases} \lbrace
 k \mid 1 <k < \min \{ m,n \} \text { and } k \in \Z^{-\sigma} \rbrace \text{ if } mn >0 , m > 0 \\
\lbrace k \mid \max \{ m,n \}  <k < 0\text { and } k \in \Z^{-\sigma} \rbrace \text{ if } mn >0 ,m < 0\\
 \phi  \text{ if } mn \leq 0\end{cases} 
\end{align*}
 and $\mathcal{C}^1_B(\what G)_{m,n}$ is the set of all functions $F: \Z^1_{m,n} \rightarrow \C $. Then from \cite[Theorem 18.2]{Barker} we have the following result:
\begin{lemma}\label{Isomorphism of schwarz space}
The Fourier transform, $f \mapsto (\what {f}_H, \what f _B)$ is a topological isomorphism between $\mathcal{C}^1(G)_{m,n} $ and $\mathcal{C}^1(\what G)_{m,n} = \mathcal{C}^1_H (\what G)_{m,n} \times \mathcal{C}^1_B(\what G)_{m,n}.$ Moreover, the restriction of that isomorphism  gives,
\begin{itemize}
\item[(i)] $\mathcal {C}^1_H(G)_{m,n}$ ismorphic to $\mathcal {C}^1_H(\what G)_{m,n}$, 
\item[(ii)] $\mathcal {C} ^1_B(G)_{m,n}$ ismorphic to $\mathcal{C}^1_B(\what G)_{m,n}$. \\
\end{itemize}
\end{lemma}
%%%%%%%%%%%%%%%%%%%%%%%%%%%%%%%%%%%%%%%%%%%%%%%%%%%%%%%%%%%%%%%%%
%Here we write payley wiener Theorems%%%%%%%%%%%%%%%%%%%%%%%%%%%%%%%%%%%%%%%%%%%%%%%%%%%%%%%%%%%%%%%%%%%%%%%%%%%%%%%%%%%%%%%%%%%%%%%%%%%%%%%%%%%%%%%%%%%%%%%%%%%%%%%%%\\

\noindent{\bf Hypergeometric function:} We are going to use the following properties of hypergeometric function,
\begin{enumerate}
\item[(a)] The hypergeometric function has the following integral representation for $\Re c>\Re b>0$,
\begin{eqnarray}\label{integral representation of hypergeometric function}
_2F_1(a,b;c;z)=\frac{\Gamma(c)}{\Gamma(b)\Gamma(c-b)}\int_0^1 s^{b-1}(1-s)^{c-b-1}(1-sz)^{-a}ds,\hspace{3mm}|z|<1.
\end{eqnarray}
(see \cite[p, 239, eqn 9.1.4]{Lebedev})
\item [(b)] \begin{eqnarray}\label{properties of hypergeometric function-2}
\nonumber c(c+1){}_2F_1(a,b;c;z)&=& c(c-a+1)_2F_1(a,b+1;c+2;z)\\&&+
a\left[c-(c-b)z\right]{}
 _2F_1(a+1,b+1;c+2;z), \hspace{3mm} z\in\mathbb{C}\setminus[1,\infty). 
\end{eqnarray}
(see \cite[p. 240, eqn. (9.1.7)]{Lebedev})

\item[(c)] \begin{eqnarray}\label{properties of hypergermetric function-3}
\nonumber \int_0^1 x^{d-1}(1-x)^{b-d-1}{}_2F_1(a,b;c;x)dx &=&\frac{\Gamma(c)\Gamma(d)\Gamma(b-d)\Gamma(c-a-d)}{\Gamma(b)\Gamma(c-a)\Gamma(c-d)},\\
&&\textup{if}\hspace{1mm}\Re d>0, \Re (b-d)>0,\Re(c-a-d)>0.
\end{eqnarray}
(see \cite[p. 813, 7.512 (3)]{Grad})
\end{enumerate}

\section*{\textbf{The functions $b_\lambda$ : Representatives of $B_\lambda, \Re\lambda>1$}}
In this section first we find the expression of $\phi_{\sigma,\lambda}^{n,n}$ and the second solutions of \eqref{Equation of eigen function of Pi {n,n}} in terms of hypergeometric functions. We substitute 
\begin{equation*}
 f(t)=g(t) \cosh^n t,
\end{equation*}in the equation \eqref{Equation of eigen function of Pi {n,n}}. Then we get the following ODE,
$$\frac{d^2g }{dt^2} + ((2n+1)\tanh t + \coth t ) \frac{dg}{dt}   + ({(n+1)}^2- {\lambda}^2)g = 0, \, t>0 .$$
By the change of variable $z:= -\sinh^2 t $  the equation reduces to the following hypergeometric differential equation 
\begin{equation}
z(1-z) \frac{d^2g}{dz^2}  + (c-(1+a+b)z) \frac{dg}{dz} - \frac{1}{4} ab g=0 \label{Hypergeometric equation reduced},
\end{equation}
with $a=\frac{n+1}{2}+ \frac{\lambda}{2}, b=\frac{n+1}{2}- \frac{\lambda}{2}, c=1 $. Therefore 
$$g_1(t) =  {}_2F_1 \left(\frac{n+1}{2}+ \frac{\lambda}{2},\frac{n+1}{2}- \frac{\lambda}{2};1;-\sinh^2 t\right) ,$$
is a solution of \eqref{Hypergeometric equation reduced} which is regular at origin, so by the uniqueness of regular solution, $$\phi_{\sigma,\lambda}^{n,n}(a_t)=(\cosh t)^n \, {}_2 F_1 \left(\frac{n+1}{2}+ \frac{\lambda}{2},\frac{n+1}{2}- \frac{\lambda}{2};1;-\sinh^2t\right), t>0.$$    Also from \cite[p.105, 2.9 (11)]{Erdelyi-1} second solutions of \eqref{Hypergeometric equation reduced}  on $(0, \infty)$ are,
$$g_2(t)= (\cosh t)^{-(1+\lambda + n)} \, {}_2F_1\left( \frac{1+\lambda}{2}+ \frac{n}{2}, \frac{1+\lambda}{2}- \frac{n}{2} ; 1+ \lambda; \cosh^{-2}t \right), $$
$$g_3(t)= (\cosh t)^{-(1-\lambda + n)}  \, {}_2F_1\left( \frac{1-\lambda}{2}+ \frac{n}{2}, \frac{1-\lambda}{2}- \frac{n}{2} ; 1- \lambda; \cosh^{-2}t \right) .$$
We now define,\begin{equation}
{\Phi_{\sigma,\lambda}^{n,n}(a_t)} = (2\cosh t)^{-(1+\lambda )} \, {}_2F_1\left( \frac{1+\lambda}{2}+ \frac{|n|}{2}, \frac{1+\lambda}{2}- \frac{|n|}{2} ; 1+ \lambda; \cosh^{-2}t \right).
 \end{equation}
Then ${\Phi_{\sigma,\lambda}^{n,n}}$ and ${\Phi_{\sigma,-\lambda}^{n,n}}$ both are  solutions of \eqref{Equation of eigen function of Pi {n,n}}, both has singularity at $t=0$ and they are linearly independent.  For $\lambda \in \C \setminus \Z,$  from \cite[p.110, 2.10(2,3 and 5)]{Erdelyi-1} we have,
  \begin{equation}
  {\phi_{\sigma,\lambda}^{n,n}} = c^{n,n}_\sigma (\lambda) \Phi_{\sigma,\lambda}^{n,n} + c^{n,n}_\sigma(-\lambda){\Phi_{\sigma,-\lambda}^{n,n}},
  \end{equation}

\noindent where  $\sigma$ is determined by $n \in \Z^\sigma$ and
 \begin{equation}
 c^{n,n}_\sigma (\lambda)=\dfrac{2^{1+\lambda} \Gamma(-\lambda)	}{\Gamma(\frac{1-\lambda-|n|}{2})\Gamma(\frac{1-\lambda+|n|}{2})}\,.\label{definition of cnn (lambda)}
 \end{equation}
  We have for $t \rightarrow \infty$,
\begin{equation}
{\Phi_{\sigma,\lambda}^{n,n}(a_t)}= e^{-(\lambda+1)t}(1+O(1)).\\
\end{equation}

\noindent Hence for $\Re \lambda <0$ and as $t \rightarrow \infty,$
\begin{equation}\label{estimate of phi lambda}
{\phi_{\sigma,\lambda}^{n,n}}(a_t) =c^{n,n}_\sigma(\lambda)e^{-(\lambda+1)t}(1+O(1)).
\end{equation}
 For simplicity if $f \in L^1(G)_{n,n}$ then we denote the principal and discrete parts of the Fourier transform by $\what{f_H}  $ and $\what{f_B} $ respectively. %(instead of $(\what{f_H}(\lambda))_{n,n}$  and $(\what{f_B}(k))_{n,n}$ ).  
 Since $n \in \Z^\sigma$ determines $\sigma$ we will use $c^{n,n} (\lambda)$ instead of $c^{n,n}_\sigma(\lambda)$.
 
Let ${\C}_+= \lbrace z \in \C \mid \Re(z) > 0 \rbrace$ be the open right half plane and $\textbf{B} = \{ \lambda \in \Z^{-\sigma} : 0< \lambda < n \}$ are the zeros of $c^{n,n}(-\lambda)$ in the open right half plane.

\noindent  For $\lambda \in \C_+ \setminus \textbf{B} $, we define
\begin{equation}\label{definition of b lambda}
b_\lambda (a_t) := \dfrac{1}{2 \lambda c^{n,n}(-\lambda)} {\Phi_{\sigma,\lambda}^{n,n}}(a_t), \quad \textup {for }   t> 0.
\end{equation} 
 \noindent Then for $\xi \in \R $ and $\xi > n+1$, $b_\xi (a_t)$ is positive. Now by Cartan decomposition we extend $b_\lambda $ as a $(n,n)$ type function on $G \setminus K $ and so  $b_\lambda $ is a solution of \eqref{Omega equation} on  $G \setminus K $.   \\
 Later on we shall need estimates of $b_\lambda(a_t)$ near $t= 0 $ and $t=\infty$. For this purpose, we need the following lemma.
 \begin{lemma}\label{lemma- estimates of b lambda} Let  $\textbf{B}_1 = \bigcup_{i=0}^{k_0 - 1} B(|n|-2i-1 ; 1) $ where $B(z ; 1)$ denotes a Euclidean ball of radius $1$ centered at $z$ and $k_0 = [\frac{|n|}{2}]+1$. Then for $\lambda \in \C_+ \setminus \textbf{B}_1 ,$ we have % $b_\lambda$ satisfies the following estimates near $0$ and $\infty$.
 
 \begin{enumerate}
\item[(i)] There is a positive constant $C$ independent of $\lambda$ such that for all $t\in(0,1/2]$,
$$|b_\lambda(a_t)|\leq C  \log \frac{1}{t}.$$
 \item[(ii)] There is a positive constant $C$ independent of $\lambda$  such that for all $t\in [1/2,\infty)$,
$$|b_\lambda(a_t)|\leq C e^{-(\Re\lambda+1)t}.$$
\end{enumerate} 
\end{lemma} 
\begin{proof} 
(i) Since $${\Phi_{\sigma,\lambda}^{n,n}(a_t)} = (2\cosh t)^{-(1+\lambda )}  {}_2F_1\left( \frac{1+\lambda}{2}+ \frac{|n|}{2}, \frac{1+\lambda}{2}- \frac{|n|}{2} ; 1+ \lambda; \cosh^{-2}t \right),$$  we first find the estimate of the hypergeometric function near $t=0$ and then by polynomial approximation of gamma functions we will finally find the estimate of $b_\lambda$.
Now by \eqref{properties of hypergeometric function-2},
\begingroup
\allowdisplaybreaks
\begin{align}
&{}_2F_1\left( \frac{1+\lambda + |n|}{2}, \frac{1+\lambda -|n|}{2} ; 1+ \lambda; \cosh^{-2}t \right) \\&
=\frac{1}{(1+\lambda)(2+\lambda)} \Bigg[ (1+\lambda) \left( \frac{1+\lambda-|n|}{2}+1 \right) {}_2F_1\left( \frac{1+\lambda + |n|}{2}, \frac{1+\lambda-|n|}{2} + 1 ; 1+ \lambda +2; \cosh^{-2}t \right) \notag \\&
 \quad + \frac{1+\lambda +|n|}{2} \left( (1+ \lambda) - \frac{1+ \lambda + |n|}{2} \cosh^{-2} t \right) \, {}_2F_1 \left( \frac{1+\lambda+|n|}{2}+ 1 , \frac{1+\lambda-|n|}{2} + 1 ; 1+ \lambda +2; \cosh^{-2}t \right) \Bigg] \notag
\end{align}
\endgroup
Since $(1+ \lambda)(2+ \lambda) $ has no zero in $ \C_+$  and $\cosh t \asymp 1$ near $t=0$, so for all $t \in (0,\frac{1}{2}]$ we have $ \dfrac{(1+\lambda) (\frac{1+\lambda-|n|}{2}+1)}{(1+\lambda)(2+\lambda)} \asymp C$  and  $\dfrac{\left( (\frac{1+\lambda +|n|}{2}) \left( (1+ \lambda) - \dfrac{1+ \lambda + |n|}{2} \cosh^{-2} t \right) \right) } {(1+\lambda ) (2+ \lambda) } \asymp C$ for all $ \lambda \in \C _+$. 

By the same argument and applying the formula \eqref{properties of hypergeometric function-2} $k_0 = [\frac{|n|}{2}] +1 $ times  we can write ,
\begin{align*}
&{}_2F_1\left( \frac{1+\lambda + |n|}{2}, \frac{1+\lambda -|n|}{2}, ; 1+ \lambda; \cosh^{-2}t \right)\\
& = \sum_{i=0}^{k_0} \frac{P_i(\lambda, \cosh^{-2}t)}{Q_i(\lambda)} {}_2F_1\left( \frac{1+\lambda + |n|}{2} +i, \frac{1+\lambda -|n|}{2} + k_0 ; 1+ \lambda + 2k_0; \cosh^{-2}t \right)
\end{align*}
where  $P_i$'s are  polynomials in $\lambda$ and $\cosh^{-2}t$ and $Q_i$'s are polynomials in $\lambda$ which has no zero in $\C_+$ such that for all  $t \in  (0,\frac{1}{2}]$,
 $$\dfrac{P_i(\lambda, \cosh^{-2}t)}{Q_i(\lambda)}  \asymp C$$ for all $\lambda \in \C_+$.
 % So, 
%%\begin{align*}
%%&{}_2F_1\left( \frac{1+\lambda + n}{2}, \frac{1+\lambda -n}{2}, ; 1+ \lambda; %%\cosh^{-2}t \right)\\
%%& \leq C \sum_{i=0}^{k_0} {}_2F_1\left( \frac{1+\lambda + n}{2} +i, %%frac{1+\lambda -n}{2} + k_0, ; 1+ \lambda + 2k_0; \cosh^{-2}t \right)
%end{align*}%
\noindent Now since $\Re (1+ \lambda + 2k_0 ) > \Re (\frac{1+\lambda -|n|}{2} + k_0)$  from (\ref{integral representation of hypergeometric function}),
\begin{align}
&_2F_1\left( \frac{1+\lambda + |n|}{2} +i, \frac{1+\lambda -|n|}{2} + k_0; 1+ \lambda + 2k_0; \cosh^{-2}t \right) \notag \\
  =& \, \mathcal{C}(\lambda;n) \int_0^1 s^{\frac{1+\lambda -|n|}{2} + k_0 -1} {(1-s)}^{\frac{1+\lambda +|n|}{2}+ k_0 -1} {(1-s \cosh^{-2}t)}^{ -\frac{1+\lambda +|n|}{2} -i} ds \notag\\
 &\textup{   where  $\mathcal{C}(\lambda;n) = \frac{ \Gamma(1+ \lambda + 2k_0)}{ \Gamma \left( \frac{1+\lambda -|n|}{2} + k_0 \right) \Gamma \left( \frac{1+\lambda +|n|}{2}+ k_0 \right) }. $   \notag}\\
= & \, \mathcal{C}(\lambda;n)  (\cosh t)^{1+ \lambda + |n| + 2i}    \int_0^1 s^{\frac{1+\lambda -|n|}{2} + k_0 -1} {(1-s)}^{\frac{1+\lambda +|n|}{2}+ k_0 -1} {(\cosh^{2}t-s )}^{ -\frac{1+\lambda +|n|}{2} -i} ds \notag\\
&\textup{( writing $\cosh^2 t = 1+x $ and making the change of variable $ s \rightarrow 1-s$ we get) \notag }\\
\, \, \,
=& \, \mathcal{C}(\lambda;n) (\cosh t)^{1+ \lambda + |n| + 2i}\int_0^1 {(1-s)}^{\frac{1+\lambda -|n|}{2} + k_0 -1} {s}^{\frac{1+\lambda +|n|}{2}+ k_0 -1} {(x+ s )}^{ -\frac{1+\lambda +|n|}{2} -i} ds.
\end{align} 
Let $I $ be the integration above. Then,
$$|I| \leq \int_0^1 {(1-s)}^{\frac{1+\Re \lambda -|n|}{2} + k_0 -1} {s}^{\frac{1+ \Re \lambda +|n|}{2}+ k_0 -1} {(x+ s )}^{ -\frac{1+ \Re \lambda +}{2} -i} ds.$$
Now we let  $I_1$ be the integration on $(0, \frac{1}{2}]$ and  $I_2$ on $( \frac{1}{2}, 1]$. Then,
\begin{align*}
I_2 \leq C \int_{\frac{1}{2}}^1 (1-s)^{\frac{-1-|n|}{2} +k_0} 2^{\frac{-\Re \lambda}{2}} 2^{\frac{1+\Re \lambda +|n|}{2} +i} ds \leq C \text{(independent of $\lambda$)},
\end{align*}
and \begin{align*}
I_1 & \leq C\int_0^\frac{1}{2} {s}^{\frac{1+ \Re \lambda +|n|}{2}+ k_0 -1} {(x+ s )}^{ -\frac{1+ \Re \lambda +|n|}{2} -i}ds\\
& \leq C    {\dfrac{2^{ - {\frac{1+ \Re \lambda + |n|}{2} -k_0}} {(x+\frac{1}{2})^{- \frac{1+ \Re \lambda +|n|}{2} -i}}}{ \frac{1+ \Re \lambda +|n|}{2} +k_0}}  + C \dfrac{\frac{1+ \Re \lambda +|n|}{2} +i}{\frac{1+ \Re \lambda +|n|}{2} +k_0} \int_0^{\frac{1}{2}} s^{\frac{1+ \Re \lambda +|n|}{2}+ k_0} (x+s)^{- (\frac{1+ \Re \lambda +|n|}{2} +i)-1} ds\\
& \leq C + C \int_0^\frac{1}{2} {\left(  \frac{s}{x+s}\right)}^{\frac{1+ \Re \lambda +|n|}{2} +i} \frac{1}{x+s} ds\\
& \leq C + C \log \left(1+ \frac{1}{2x} \right).
\end{align*} Since $x= \sinh^2 t $ and $\log$ is an increasing function we have, $I_1 \leq C \log \frac{1}{t}$,  hence it follows $|I| \leq C \log \frac{1}{t}$ for all $ t \in (0,\frac{1}{2}]$.\\

 \noindent We now turn to the estimate of $b_\lambda$. Using the inequality above of $|I|$ and applying the expression of $c^{n,n}(\lambda)$ in  (\ref{definition of b lambda}) we get,
\begin{align*}
|b_\lambda(a_t)| & \leq  C \left|\dfrac{\Gamma \left( \dfrac{1+\lambda-|n|}{2} \right)\Gamma(\frac{1+\lambda+|n|}{2}) \Gamma(1+ \lambda + 2k_0)} {   \Gamma(1 + \lambda)  \Gamma(\frac{1+\lambda -|n|}{2} + k_0) \Gamma (\frac{1+\lambda +|n|}{2}+ k_0) }   \right| \log\frac{1}{t}\\
& \leq C \left|\dfrac{\Gamma(\frac{1+\lambda-|n|}{2} +k_0)\Gamma(\frac{1+\lambda+|n|}{2}) \Gamma(1+ \lambda + 2k_0)} {   (\frac{1+\lambda-|n|}{2})(\frac{1+\lambda-|n|}{2} +1) \cdots  (\frac{1+\lambda-|n|}{2}+ k_0 -1) \Gamma(1 + \lambda)  \Gamma(\frac{1+\lambda -|n|}{2} + k_0) \Gamma (\frac{1+\lambda +|n|}{2}+ k_0) }   \right| \log\frac{1}{t}\\  
& \leq C \left| \frac{(1+ |\lambda|)^{k_0} }{(\frac{1+\lambda-|n|}{2})(\frac{1+\lambda-|n|}{2} +1) \cdots  (\frac{1+\lambda-|n|}{2}+ k_0 -1)} \right| \log \frac{1}{t}
\end{align*} 
The last line of the inequality above follows from  \cite[Appendix, Lemma A.3 ]{PS}   and the fact that $k_0 - \frac{|n|}{2} \geq \frac{1}{2}$.
Therefore for all $\lambda \in \C_+ \setminus \textbf{B}_1$, (where $\textbf{B}_1 =\bigcup_{i=0}^{k_0 - 1} B(|n|-2i-1 ; 1) $ )
\begin{equation}
|b_\lambda(a_t)| \leq C \log \frac{1}{t},
 \end{equation}
 for all $t \in (0, \frac{1}{2}]$ and $C$ is independent of $\lambda$.

(ii) Since $\Phi^{n,n}_{\sigma,\lambda}(a_t) \asymp e^{-(\Re \lambda +1)t}$ near $\infty$ and by the definition of $b_\lambda$ and $c^{n,n}(\lambda) $ we get for all  $t \in [\frac{1}{2}, \infty),$\\
%the expression of $c^{n,n}(\lambda) $ we get for all  $t \in [\frac{1}{2}, \infty),$\\
\begin{align*}
|b_\lambda(a_t)|  & \leq   C \left| \dfrac{\Gamma(\frac{1+\lambda-|n|}{2})\Gamma(\frac{1+\lambda+|n|}{2})}{\Gamma(1 + \lambda)}\right| |\Phi_{\sigma,\lambda}^{n,n}(a_t)| \\
 \leq  &  C  \left| \dfrac{\Gamma(\frac{1+\lambda-|n|}{2} +k_0)\Gamma(\frac{1+\lambda+|n|}{2})}{(\frac{1+\lambda-|n|}{2})(\frac{1+\lambda-|n|}{2} +1) \cdots  (\frac{1+\lambda-|n|}{2}+ k_0 -1) \Gamma(1 + \lambda)}\right| \, e^{-(\Re \lambda +1)t} \\
 \leq & C \left| \dfrac{(1+|\lambda|)^{k_0- \frac{1}{2}}}{(\frac{1+\lambda-|n|}{2})(\frac{1+\lambda-|n|}{2} +1) \cdots  (\frac{1+\lambda-|n|}{2}+ k_0 -1)} \right| \, e^{-(\Re \lambda +1)t} \hspace*{3mm}  \text{(using   \cite[Appendix, Lemma A.3 ]{PS})} \\
 \leq & C \, e^{-(\Re \lambda +1)t} \hspace*{5mm}  \text{(Since $\lambda \notin \textbf{B}_1)$.}
\end{align*}
\qquad  Therefore for all $t \in [\frac{1}{2}, \infty),$ $|b_\lambda(a_t)|\leq C e^{-(\Re\lambda+1)t}$ where $C$ is independent of $\lambda$. 

\end{proof}

\begin{remark}
The proof of the lemma  above  shows that to get the estimate of $b_\lambda$ near $0$ and $\infty$ we only need to remove a neighbourhood of the zeros of $ c^{n,n}(-\lambda)$ and origin (when $n$ is odd). If we only remove the zeros of $ c^{n,n}(-\lambda)$ and origin but not the neighbourhoods then on both cases the constant on the right hand side will depend on $\lambda$, for example for $t \in (0,\frac{1}{2}],$  $|b_\lambda(a_t)| \leq C_\lambda \log \frac{1}{t} $.

\end{remark}

\begin{lemma}\label{when b lambda is in L1}

\begin{enumerate}
\item [(a)] For all $\lambda\in\C_+ \setminus \textbf{B}$, $b_\lambda$ is locally integrable at $e$.
\item [(b)] For $\Re \lambda> 1 $ and $\lambda \not\in \textbf{B} $, $b_\lambda\in L^1(G)_{n,n}$.
\item[(c)]For all $\lambda\in\C_+$ and $\lambda \not\in \textbf{B}$, $b_\lambda$ is in $L^2$ outside neighbourhood of $e$. 
\item[(d)]For each $\lambda\in\C_+ \setminus \textbf{B} $, there exists $p<2$ (depending on $\lambda$) such that $b_\lambda$ is in $L^p$ outside neighbourhood of $e$. 
\end{enumerate}
\end{lemma}
\begin{proof}
Proof of this lemma follows  directly from previous Lemma \ref{lemma- estimates of b lambda} and the asymptotic behaviour of $\Delta(t)$ near $0$ and $\infty$.
\end{proof}
\begin{remark}\label{existence of principal transform of b lambda}
By the lemma  above  $b_\lambda$ can be written as a sum of $L^1$ and $L^p$ ($p<2$) functions on $G$. Therefore its principal part of the Fourier transform  is a continuous function on $\C$, vanishing at infinity in $\C$. In fact in the next lemma we are going to find the Fourier transforms  of $b_\lambda$.
\end{remark}

\begin{lemma} \label{fourier transform of b lambda}
Let $\lambda\in \C_+ \setminus \textbf{B}$. Then,
\begin{align*}
\widehat{b}_{\lambda H}(i \xi) & =\frac{1}{\lambda^2 + \xi^2}, \, \text{ for all } \xi\in \mathbb{R} \text{ and }\\
\widehat{b}_{\lambda B}(k)&=\frac{1}{\lambda^2 -k^2}, \, \text{ for all } k \in \Gamma_n.
\end{align*} 
\end{lemma}
\begin{proof}
For two smooth functions $f$ and $g$ on $(0,\infty)$, we define 
 
$$[f,g](t)=\Delta(t)\left[ f^\prime(t)g(t) - f(t)g^\prime(t)\right],\,\,\,t>0.$$
  
\noindent An easy calculation shows that $[f,g]^\prime(t)=[\Pi_{n,n}(\Omega) f\cdot g-f\cdot \Pi_{n,n}(\Omega) g](t)\Delta(t)$. Therefore, for any $b>a>0$, we have
  \begin{eqnarray}
\int_a^b (\Pi_{n,n}(\Omega)f\cdot g-f\cdot \Pi_{n,n}(\Omega)g)(t)\Delta(t)=[f,g](b)-[f,g](a).\label{relation of integration}
\end{eqnarray}
Then by similar calculations in \cite[Lemma 8.1]{PS}  we have the following two results,
 $$ [\phi_{\sigma,\lambda}^{n,n},\Phi_{\sigma,\lambda}^{n,n}](\cdot) = 2 \lambda c^{n,n}(-\lambda)$$
and if $f$ is an even smooth function on $\mathbb{R}$ then
\begin{equation}\label{1.11}
\lim_{t \rightarrow 0^+} [f, \Phi_{\sigma,\lambda}^{n,n}](t) = 2\lambda c^{n,n}(-\lambda)f(0).
\end{equation}
%To prove our claim it is enough to prove for $f$ to be compactly supported. Then with this $f$ and $g=\Phi_{\sigma,\lambda}^{n,n}$, the equation ( \ref{relation of integration}) (for large $b$ and $a=t\rightarrow 0^+$) implies that $\lim\limits_{t \rightarrow 0^+}[f,\Phi_{\sigma,\lambda}^{n,n}](t)$ exists. Then 
%
%\begin{eqnarray*}
%\lim_{t\rightarrow 0^+}[f,\Phi_{\sigma,\lambda}^{n,n}](t)=\lim_{t\rightarrow 0^+}
%\Delta(t)\left(\Phi_{\sigma,\lambda}^{n,n}(t)\right)^2
%\left(\frac{f}{\Phi_{\sigma,\lambda}^{n,n}}\right)^\prime(t)= C \lim_{t \rightarrow 0^+} t(\log\frac{1}{t})^2 \left(\frac{f}{\Phi_{\sigma,\lambda}^{n,n}}\right)^\prime(t)
%\end{eqnarray*}
%since $\Delta(t)\asymp t$ and $\Phi_{\sigma,\lambda}^{n,n}(t) \asymp \log \frac{1}{t}$ near 0.\\
%As $f$ is smooth and $f(0)= 0$ so $ \lim\limits_{t \rightarrow 0^+} \dfrac{\left( \dfrac{f}{\Phi_{\sigma,\lambda}^{n,n}} \right)(t)}{\dfrac{1}{\log\frac{1}{t}}} = 0 $,
%therefore by again an application of L'Hospital Rule we have, $$  \lim_{t \rightarrow 0^+} t(\log\frac{1}{t})^2 \left(\frac{f}{\Phi_{\sigma,\lambda}^{n,n}}\right)^\prime(t)= 0$$
%Now if $f(0) \neq 0$ then writing,
%$$[f, \Phi_{\sigma,\lambda}^{n,n}](t)=
%[f-f(0)\phi_{\sigma,\lambda}^{n,n},\Phi_{\sigma,\lambda}^{n,n}]+f(0)[\phi_{\sigma,\lambda}^{n,n},\Phi_{\sigma,\lambda}^{n,n}]
%$$
%We get 
\textbf{CASE 1} : $\widehat{b}_{\lambda H}(i \xi)=\frac{1}{\lambda^2 + \xi^2}, \quad \text{ for all } \xi\in \mathbb{R}$.\\
For $\xi \in \R ,$ we put $f= \phi_{\sigma,i\xi}^{n,n}, g=\Phi_{\sigma,\lambda}^{n,n}$
in equation (\ref{relation of integration}) we get,
$$\int_a^b \Phi_{\sigma,\lambda}^{n,n} (t) \phi_{\sigma,i\xi}^{n,n}(t) \Delta(t) dt = \frac{1}{-\lambda^2-\xi^2}\left([\phi_{\sigma,i\xi}^{n,n},\Phi_{\sigma,\lambda}^{n,n}](b)-
[\phi_{\sigma,i\xi}^{n,n},\Phi_{\sigma,\lambda}^{n,n}](a)\right).$$
Taking $a \rightarrow 0^+$, we get from (\ref{1.11})
$$\int_0^b \Phi_{\sigma,\lambda}^{n,n} (t) \phi_{\sigma,i\xi}^{n,n}(t) \Delta(t) dt= \dfrac{2\lambda c^{n,n}({-\lambda)}}{\lambda^2 + \xi^2} - \dfrac{[\phi_{\sigma,i\xi}^{n,n},\Phi_{\sigma,\lambda}^{n,n}](b)}{\lambda^2 + \xi^2}.$$
Therefore if we could show $ [\phi_{\sigma,i\xi}^{n,n},\Phi_{\sigma,\lambda}^{n,n}](b) \rightarrow 0$ as $b \rightarrow \infty$ then we will be done.\\
\noindent We note that the existence of limit is guaranteed by the equation  above. As like before we can write,
$$\lim_{b \rightarrow \infty}[\phi_{\sigma,i\xi}^{n,n},\Phi_{\sigma,\lambda}^{n,n}](b)= \lim_{b \rightarrow \infty} e^{-2 \lambda b}\left(\dfrac{\phi_{\sigma,i\xi}^{n,n}}{\Phi_{\sigma,\lambda}^{n,n}}\right)^\prime(b).$$
By the asymptotic behavior of $\phi_{\sigma,i\xi}^{n,n}$ and $\Phi_{\sigma,\lambda}^{n,n}$,
$$
\lim_{b\rightarrow\infty}\dfrac{\dfrac{\phi_{\sigma,i\xi}^{n,n}}{\Phi_{\sigma,\lambda}^{n,n}}(b)}{e^{2\lambda b}}=0.$$
Finally for all $\lambda \in \C_+ \setminus \textbf{B},$
$$\widehat{b}_{\lambda H}(i \xi) = \dfrac{1}{\lambda^2 -(i \xi)^2}, \, \text { for all } \xi \in \R .  $$\\

\textbf{CASE 2} : $\widehat{b}_{\lambda B}(k)=\dfrac{1}{\lambda^2 -k^2} \quad \text{ for all } k \in \Gamma_n$.\\
 We note that from \cite[p.30 propn 7.3]{Barker} we have ${\psi_{k}^{n,n}} ={\phi_{\sigma,|k|}^{n,n}}$ for all $k \in \Gamma_n$.\\
Let $k \in \Gamma_n$, we put $f= \psi_{k}^{n,n}, g=\Phi_{\sigma,\lambda}^{n,n}$
in equation (\ref{relation of integration}) to get,
$$\int_a^b \Phi_{\sigma,\lambda}^{n,n} (t) \psi_{k}^{n,n}(t) \Delta(t) dt = \frac{1}{-\lambda^2+k^2}\big([\psi_{k}^{n,n},\Phi_{\sigma,\lambda}^{n,n}](b)-
[\psi_{k}^{n,n},\Phi_{\sigma,\lambda}^{n,n}](a)\big).$$
Taking $a \rightarrow 0^+$, we get from (\ref{1.11})
$$\int_0^b \Phi_{\sigma,\lambda}^{n,n} (t) \psi_{k}^{n,n}(t) \Delta(t) dt= \dfrac{2\lambda c^{n,n}({-\lambda)}}{\lambda^2 - k^2} - \dfrac{[\psi_{k}^{n,n},\Phi_{\sigma,\lambda}^{n,n}](b)}{\lambda^2 - k^2}.$$
Therefore if we could show $ [\psi_{k}^{n,n},\Phi_{\sigma,\lambda}^{n,n}](b) \rightarrow 0$ as $b \rightarrow \infty$ then we will be done.\\
%%%%%%%%%%%%%%%%%%%%%%%%%%%%%%%
\noindent From \cite[p.33 Theorem 8.1]{Barker} we get that there exist constants $C, r_1, r_2, r_3  \geq 0$ such that\\
$$|\psi^{n,n}_k( t)| \leq C (1+ |n|)^{r_1} (1+|k|)^{r_2} (1+t)^{r_3} e^{-2t}$$ for all $k \in \Z^*$ for which $|k| \geq 1$ and for all $n \in \Z (k).$ \\

\noindent Now by the asymptotic behaviour of $\Phi_{\sigma,\lambda}^{n,n}$ we get,

$$\left| \dfrac{\psi_k^{n,n}(b)}{e^{2 \lambda b} \Phi_{\sigma,\lambda}^{n,n}(b)} \right| \leq  C \dfrac{ (1+|k|)^{r_2} (1+b)^{r_3} e^{-2b}}{e^{2\lambda b}  e^{-(\lambda +1)b}} \leq C \dfrac{ (1+|k|)^{r_2} (1+b)^{r_3} }{e^{(\lambda+1)b}}$$
for a fixed $n$. Therefore  $\lim\limits_{b \rightarrow \infty } \left| \dfrac{\psi_k^{n,n}(b)}{e^{2 \lambda b} \Phi_{\sigma,\lambda}^{n,n}(b)} \right| = 0$. This completes the proof.

\end{proof}

\begin{remark}\label{fourier transform of b lambda in S1}
Since for $\Re \lambda > 1 $ and $\lambda \not\in \textbf{B}$, $b_\lambda$ is in $L^1(G)_{n,n}$ and its principal Fourier transform is a well defined continuous function on the strip $S_1$, which is also holomorphic in $S_1^o$. Therefore by analytic continuation we can write for $\Re \lambda > 1 $ and $\lambda \not\in \textbf{B}$,
$$\widehat{b_\lambda}_H (z) = \dfrac{1}{\lambda^2 - z^2}, \, \text{ for all } z \in S_1.$$
\end{remark}
We now turn to the estimates of $||b_\lambda||_1$ which is essential in $\S$\ref{Resolvent Transform}.
\begin{lemma}\label{L-1 norm estimates of b-lambda}
\noindent (i) If $\Re\lambda>1$ and $\lambda \not\in \textbf{B}_1$ ,
$
||b_\lambda||_1\leq C\frac{(1+|\lambda|)}{\Re\lambda-1}
$ for some $C>0$.

(ii) $||b_\lambda||_1\rightarrow 0$ if $\lambda\rightarrow \infty$ along the positive real axis.
\end{lemma}
\begin{proof}
(i)  Since $\Delta(t)\asymp t$ near $0$ and $\Delta(t)\asymp e^{2 t}$ near $\infty$, from Lemma \ref{lemma- estimates of b lambda} we can write,
\begin{align*}
||b_\lambda||_1 &= \int_0^{1/2}|b_\lambda(a_t)|\Delta(t)dt + \int_{\frac{1}{2}}^\infty |b_\lambda(a_t)|\Delta(t)dt  \\
&\leq \int_0^{\frac{1}{2}} t \log\frac{1}{t} + C \int_{\frac{1}{2}}^\infty e^{( 1-\Re \lambda )t} \, \\
&\leq C + \frac{C}{\Re \lambda-1} \, \\
& \leq C \, \dfrac{1+ |\lambda|}{\Re \lambda -1} .
\end{align*}
%and
%$$
%\int_{\frac{1}{2}}^\infty |b_\lambda(a_t)|\Delta(t)dt \leq C \int_{\frac{1}{2}}^\infty e^{( 1-\Re \lambda )t} \leq   \frac{C}{\Re \lambda-1} \leq C \dfrac{1+|\lambda|}{\Re \lambda-1} . $$
%Hence the result follows.\\%

(ii)If $\lambda = \xi \in \R $ and $\xi > n+1$ then $b_\xi (a_t)$ is nonnegative.  Hence
\begin{align*}
||b_\xi||_1 = \int_{\R_+} b_\xi (a_t) \Delta(t) dt &\leq  \int_{\R_+} (\cosh t)^{|n|} b_\xi (a_t) \Delta(t) dt  \quad  \\
&= \frac{1}{\xi^2-(n+1)^2}. 
\end{align*}
The last line of the inequalities follows from similar calculation of \cite[Lemma 3.3]{PS}, which uses \eqref{properties of hypergermetric function-3}.
%%%%%%%%%%%%%%%%%%%%%%%%%%%%%%%%%%%%%%%%%%%%%%%%
Hence the proof follows.
\end{proof}
\begin{lemma}\label{b lambda dense in L1(G)n,n }

The functions $ \{ b_\lambda  \mid \Re \lambda > 1 $ and $\lambda \not\in \textbf{B} \}$ span a dense subset of $L^1(G)_{n,n}$.
\end{lemma}
\begin{proof}
We will show that $\overline{\text{span} \{ b_\lambda | \Re \lambda >1 \text{ and } \lambda \notin \textbf{B} \} }$
 contains $C_c^\infty (G)_{n,n}$ and since $C_c^\infty (G)_{n,n}$ is dense in $L^1(G)_{n,n}$, the lemma will follow.\\
 Let $f\in C_c^\infty(G)_{n,n}$. Since $\what{f}_H$ is entire and it has polynomial decay on any bounded vertical strip (by Paley-Wiener theorem) Cauchy's formula implies that
 \quad $$\widehat{f}_H(w)=\dfrac{1}{2\pi i } \int_{\Gamma_1 } \dfrac{ \widehat{f}_H(z)}{z-w} dz  + \dfrac{1}{2\pi i } \int_{\Gamma_2 } \dfrac{ \widehat{f}_H(z)}{z-w} dz, \hspace*{3mm} \text{for $w \in \C.$ } $$ 
\textup{where $\Gamma_1= (|n|+2)+ i\R $ downward and $\Gamma_2= -(|n|+2)+ i\R $ upward.} Next by the change of variable $z \rightarrow -z$ in the second integral,
\begin{align}
\widehat{f}_H(w) &= \dfrac{1}{2\pi i } \int_{\Gamma_1 } \dfrac{ \widehat{f}_H(z)}{z-w} dz + \dfrac{1}{2\pi i } \int_{\Gamma_1 } \dfrac{ \widehat{f}_H(-z)}{-z-w} (-dz) .\notag
\end{align}
We know $\what{f}_H(z)$ is an even function, therefore for all $w \in \C$ 
\begin{equation}\label{cauchy formula presentation of fH}
\what{f}_H(w)= \dfrac{1}{2\pi i } \int_{\Gamma_1 } \dfrac{ 2z \widehat{f}_H(z)}{z^2-w^2} dz.  
\end{equation}
%\noindent For $z \in \Gamma_1 , \Re z >1 $  and $z \not\in \textbf{B} $, $b_z$ is in $L^1(G)_{n,n}$ and  $$\widehat{b}_{zH} (w) =\frac{1}{z^2-w^2},  \,  w \in S_1,$$
%$$\widehat{b}_{zB} (k) =\frac{1}{z^2-k^2},  \,  k \in \Gamma_n.$$ 
Since $ \widehat{f}_B(k) =  \widehat{f}_H(|k|) \text{ for all $k \in \Gamma_n$},$ so from \eqref{cauchy formula presentation of fH} and together with Lemma \ref{fourier transform of b lambda} we get,
\begin{eqnarray}
\widehat{f}_H(w) = \dfrac{1}{2\pi i } \int_{\Gamma_1 }  2z \widehat{f}_H(z)\widehat{b}_{zH} (w)  dz, \, \text{ for all }  w \in S_1, \label{principal fourier transform of f in terms of bz}\\
\widehat{f}_B(k) = \dfrac{1}{2\pi i } \int_{\Gamma_1 }  2z \widehat{f}_H(z)\widehat{b}_{zB} (k)  dz, \,  \text{ for all }  k \in  \Gamma_n.\label{discrete fourier transform of f in terms of bz}
\end{eqnarray}

The decay condition on $\what{f}_H $ and Lemma \ref{L-1 norm estimates of b-lambda} imply that the $L^1(G)_{n,n}$ valued integral 
$$\dfrac{1}{2\pi i } \int_{\Gamma_1 }  2z \widehat{f}_H(z){b}_{z} (\cdot)  dz$$
converges and  \eqref{principal fourier transform of f in terms of bz}, \eqref{discrete fourier transform of f in terms of bz} implies that it must converge to $f$.
%Let us define $g(x) = \dfrac{1}{2\pi i } \int_{\Gamma_1 }  2z \widehat{f}_H(z){b}_{z} (x)  dz $ for all $x \in G\setm K,$ then by lemma \eqref{L-1 norm estimates of b-lambda} and the decay condition on $\what{f}_H $ imply %that $g$ is in $L^1(G)_{n,n}$. Then $\what{g}_H (z) =  \widehat{f}_H(z) $ for all $z \in S_1$ and  $\what{g}_B (k) =  \widehat{f}_B(k) $ for all $k \in  \Gamma_n$ hence% $g$ must be equal to $f$. Since we defined $g$ using integration and the%
 Thus the Riemann sums which are nothing but finite linear combinations of $b_\lambda$'s converge to $f$. So we can conclude that $f$ is in the closed subspace spanned by  $ \{ b_\lambda  \,  | \, \Re \lambda >1 \text{ and } \lambda \not\in \textbf{B}  \}   .$ The lemma follows.
\end{proof}

\section{\textbf{Resolvent transform}}\label{Resolvent Transform}

Let $L^1_\delta(G)_{n,n}$ be the unitization of $L^1(G)_{n,n} $ and $\delta$, where $\delta$ is the $(n,n)$ type distribution defined by $\delta(\phi)= \phi(e)$ for all $\phi \in C_c^\infty (G)_{n,n}$. Maximal ideal space of $L^1_\delta(G)_{n,n}$ is  $\big\{L_z:z\in S_1\cup\{\infty\}\big\}$ and $\big\{L_k^{\prime}: k \in \Gamma_n \big\}$, %which can be identified with $S_1 \cup {(\infty)} \cup \{k \, | \,  0<k<n  \, \textup{and} \, k\in \Z^{-\sigma} \}$%
 where $L_z$ and $L_k^{\prime}$ are the complex homomorphism on $L^1_\delta(G)_{n,n}$ defined by 
$$ L_z(f)= \what{f}_H(z) \text{ and } L_k^{\prime}(f) =\what{f}_B (k) \text{  for all }  f \in L^1_\delta(G)_{n,n}. $$
\noindent From now on we will denote $I$ as a closed ideal of $L^1(G)_{n,n}$ such that $\lbrace \what{f}_H : f \in I \rbrace$ and $\lbrace \what{f}_B : f \in I \rbrace$ does not have common zero on $S_1$ and $\Gamma_n$ respectively. Since $\delta \, * f = f$  for all $f \in L^1(G)_{n,n}$ so $I$ is also an ideal of $L^1_\delta (G)_{n,n}$ and $L^1_\delta (G)_{n,n}/ I$ makes sense. 

In Banach algebra theory if $J$ is a closed ideal of a commutative Banach algebra $\mathcal{A}$ then the maximal ideal space of $\mathcal{A} / J$ is 
$$\Sigma(\mathcal{A}/J) = \lbrace h \in \Sigma(\mathcal{A}) : h =0  \text{ on } J \rbrace, $$ 
 where $\Sigma(\mathcal{A})$ denotes the maximal ideal space of $\mathcal{A}.$ \\
From the theory above the maximal ideal space of $L^1_\delta (G)_{n,n}/ I$ is the complex homomorphism $\tilde{L}_\infty$ and it is defined by $$\tilde{L}_\infty(f+I) = \what{f}_H(\infty) \text { for all }f \in L^1_\delta (G)_{n,n}/ I .$$ It also follows that an element $f+ I$  in $L^1_\delta (G)_{n,n}/ I$ is invertible if and only if $\what{f}_H (\infty) \neq 0$.
%For a closed ideal $J$ we define $Z(J)$ is the set of common zeros of principal Fourier transforms of the elements $J$ (in $S_1 \cup \{\infty \} $). From now on $I$ will always stand  for  a closed ideal of $L^1(G)_{n,n}$  such that $ Z=Z(I)$ is \{$\infty$\}. It follows that $I$ is also an ideal of $L^1_\delta(G)_{n,n}$ so $L^1_\delta(G)_{n,n} /I $ makes sense. Since the discrete transforms of the elements of $I$ does not have any common zeros so the maximal ideal space of the quotient algebra  $L^1_\delta(G)_{n,n} / I$ is consists of the complex homomorphisms $ \{ \tilde{L_z} : z \in Z \}$ where $\tilde{L_z}(f+I) =\what{f}_H(z)$. Now by the Banach algebra theory an element $f+ I$ is invertible in $L^1_\delta(G)_{n,n}/ I$ if and only if $\what{f}_H(z) \neq 0$ for all $z \in Z$.%

Let $\lambda_0$ be a fixed complex number with $\Re \lambda_0 > n+1$. Then by Lemma \ref{when b lambda is in L1} $b_{\lambda_0}$ is in $L^1(G)_{n,n}$. For $\lambda \in \C $ the function, $$\lambda \mapsto \what{\delta} - (\lambda_0^2 -\lambda^2 ) \, \what{b_{\lambda_0}}_H$$ does not vanish at $\infty$ and hence ${\delta} - (\lambda_0^2-\lambda^2) \, {b_{\lambda_0}} +I$ is invertible in the quotient algebra $L^1_\delta(G)_{n,n} /I $. We put
\begin{eqnarray}\label{definition of B-lambda}
B_\lambda=\left(\delta-(
\lambda_0^2 -\lambda^2)b_{\lambda_0}+I\right)^{-1}*\left
(b_{\lambda_0}+I\right),\hspace{3mm}\text{ for }\lambda\in\C.
\end{eqnarray}
Now let $g \in  L^\infty(G)_{n,n}$ annihilates $I$,  so we can take $g $ as a bounded linear functional on $L^1(G)_{n,n} / I $. We define the resolvent transform  $\mathcal{R}[g]$ of $g$ by
\begin{eqnarray}\label{defn-R-g}
\mathcal{R}[g](\lambda)=\left\langle B_\lambda,g\right\rangle.
\end{eqnarray}
From \eqref{definition of B-lambda}, $\lambda\mapsto B_\lambda$ is a Banach space valued  even holomorphic function on $\C  $. So $\mathcal{R}[g]$ is an even holomorphic function on $\C  $.\\

We need an explicit formula of the function $\mathcal{R}[g]$ almost everywhere in $\C$. We will show for $\Re \lambda >1$ and $\lambda \notin $ \textbf{B}, $B_\lambda =b_\lambda +I$. Also, for  $0 < \Re\lambda <1$ we find a representative of the cosets $B_\lambda$ in the next section.
\section{\textbf{Representatives of} \textbf{$B_\lambda,  0 < \Re\lambda <1$ and properties of $\mathcal{R}[g]$}}
 
Let  $\lambda$ be such that $0< \Re \lambda <1$. For  $f \in L^1(G)_{n,n}$ %by Lemma \ref{when b lambda is in L1} $f*b_\lambda$ is a sum of $L^1$ and $L^p$ function so if we define the following function
we define
\begin{eqnarray}\label{definition of T-lambda-f}
T_\lambda f:=\widehat{f}_H(\lambda)b_\lambda -f*b_\lambda.
\end{eqnarray}
Since $b_\lambda$ is a sum of $L^1$ and $L^p$ functions (by Lemma \ref{when b lambda is in L1}) $T_\lambda f$ is well defined and the principal and  discrete part of Fourier transforms exist on $i\R$ and $\Gamma_n$ respectively. The proof follows directly from Lemma \ref{fourier transform of b lambda}. % and by the same reason above and Lemma \ref{when b lambda is in L1}  principal transform of $T_\lambda f$ exists on $i \R$. In the next lemma we are going to evaluate principal fourier transform  and discrete transform of $T_\lambda f$ on $i \R$ and on $\Gamma_n $  respectively which is  a direct consequence of Lemma \ref{fourier transform of b lambda}.
 %Let $f \in L^1(G)_{n,n}$. For each $\lambda$ with $0< \Re \lambda <1$, we define
 %\begin{eqnarray}\label{definition of T-lambda-f}
%T_\lambda f:=\widehat{f}_H(\lambda)b_\lambda -f*b_\lambda.
%\end{eqnarray}
%Since $b_\lambda$ can be written as a sum of $L^1$ and $L^p$ ($p<2$) function, so $T_\lambda f$ is well-defined and  $T_\lambda f$ also can be written as a sum of $L^1$ and $L^p$ ($p<2$) function. Hence the principal Fourier transform of $T_\lambda f$  is well defined and continuous on $i \R$. 

\begin{lemma}\label{Fourier transform of T lambda f}
 Let $0 < \Re\lambda <1 $ and $f$ be a $L^1(G)_{n,n}$ function on G. Then 
\begin{align*}
&\widehat{T_\lambda f}_H( i \xi)=\frac{\widehat{f}_H(\lambda)-\widehat{f}_H(i \xi)}{\lambda^2 +  \xi^2},\hspace{3mm}\textup{for all}\hspace{3mm} \xi\in\mathbb{R},\\
 & \widehat{T_\lambda f}_B(k)=\frac{\widehat{f}_H(\lambda)-\widehat{f}_B(k)}{\lambda^2  - k^2},\hspace{3mm}\textup{for all} \, k \in \Gamma_n.
\end{align*}
\end{lemma}
%\begin{proof}
%By Lemma \ref{fourier transform of b lambda} we have for all $\lambda \in \C_+ \setminus \textbf{B}$,  $\widehat{b}_{\lambda H}(i \xi) = \dfrac{1}{\lambda^2 -(i \xi)^2}\, \text{ for all } \xi \in \R   $ and $ \widehat{b}_{\lambda B}(k) = \dfrac{1}{\lambda^2 -k^2} \,  \text{ for all } k \in \Gamma_n   ,$ therefore the proof follows.
%\end{proof}
 
\begin{lemma}\label{spherical property of b lambda}
Let $\lambda \in \C_+ \setm \textbf{B} .$ Then,
\begin{eqnarray*}\int_K b_\lambda(a_ska_t) e_n(k^{-1})dk=
\begin{cases}
b_\lambda(a_s)\phi_{\sigma,\lambda}^{n,n}(a_t)\hspace{3mm}\textup{if}
\hspace{1mm}s > t\geq 0,\\
b_\lambda(a_t)\phi_{\sigma,\lambda}^{n,n} 
(a_s)\hspace{3mm}\textup{if}
\hspace{1mm}t> s\geq 0.\\

\end{cases}
\end{eqnarray*}

\end{lemma}
\begin{proof}
Since $b_\lambda$ is smooth outside K and  $a_ska_t \not\in K$  as $s \neq t$, the integral is well defined.
Fix $s > 0$, as $b_\lambda $ is a ${(n,n)}$ type eigenfunction of $\Omega$ on $G\setm K$ with eigenvalue $\frac{\lambda^2 -1}{4}$, the function \begin{equation*}
 g \mapsto \int_K b_\lambda(a_sk g) e_n(k^{-1})dk
\end{equation*} is smooth $(n,n)$ type eigenfunction of $\Omega$ on the open ball $B_s=\{k_1a_r k_2\in K \overline{A^+} K\mid r<s\}$. Hence the function $t \mapsto \int_K b_\lambda(a_ska_t) e_n(k^{-1})dk $ is a solution of \eqref{Equation of eigen function of Pi {n,n}} on $(0,s)$ which is regular at 0. Therefore,
\begin{equation*}
\int_Kb_\lambda(a_ska_t)e_n{(k^{-1})}dk=C\phi_{\sigma,\lambda}^{n,n}(a_t)\hspace
{3mm}\textup{for all}\hspace{1mm} 0\leq t<s \textup{ and for some constant C.}
\end{equation*}
Putting $t=0 $ in the equation above we get $C= b_\lambda(a_s)$. Therefore for $s> t \geq 0$ we have,
 $$\int_K b_\lambda(a_ska_t) e_n(k^{-1})dk =b_\lambda(a_s)\phi_{\sigma,\lambda}^{n,n}(a_t) \,.$$ Similarly the second case follows.
\end{proof}
Next we will show $T_\lambda f$ is in $L^1(G)_{n,n}$ and to do that we will use the following representation of $T_\lambda f$.
\begin{lemma}
Let $0 < \Re\lambda <1 $ and $f \in L^1(G)_{n,n}$. Then for all $t > 0$,
$$
T_\lambda f(a_t)=b_\lambda(a_t)\int_t^\infty f(a_s)\phi_{\sigma,\lambda}^{n,n}(a_s)\Delta(s)ds-\phi_{\sigma,\lambda}^{n,n}(a_t)
\int_t^\infty f(a_s)b_\lambda(a_s)\Delta(s)ds.
$$
\end{lemma}
\begin{proof}
Here we are going to use the fact that there exists $ k_0 \in K$ such that $k_0a_s k_0^{-1} = a_{-s}$ for all $s \geq 0.$ Now
\begin{align}
f*b_\lambda(a_t) & = \int_K \int_0^\infty \int_K f(k_1 a_s k_2)b_\lambda (k_2^{-1} a_{-s} k_1^{-1} a_t) \Delta(s) dk_1 ds dk_2 \notag \\
& =  \int_0^\infty f(a_s)  \int_K  b_\lambda (a_{-s} k_1 a_t) e_n(k_1^{-1}) \Delta(s) dk_1 ds  \hspace{3mm}(\textup{ change of variable $k_1 \rightarrow k_1^{-1}$}) \notag \\
&=\int_0^\infty f(a_s)  \int_K  b_\lambda (k_0 a_s k_0^{-1} k_1 a_t)e_n(k_1^{-1}) \Delta(s) dk_1 ds \notag \\  
%%%%%%\hspace{3mm}(\textup{ $\because \exists k_0 \in K$ such that $k_0 a_{-s} k_0^{-1} = a_s$)} %%%
&=\int_0^\infty f(a_s)  \int_K b_\lambda ( a_s k_1 a_t)  e_n(k_1^{-1}) \Delta(s) dk_1 ds \hspace{3mm}(\textup{ change of variable $k_1 \rightarrow k_0 k_1$}) \notag \\
&=\int_0^t f(a_s) b_\lambda(a_t)\phi_{\sigma,\lambda}^{n,n}(a_s) \Delta(s) ds+ \int_t^\infty f(a_s) b_\lambda(a_s)\phi_{\sigma,\lambda}^{n,n}(a_t) \Delta(s) ds \label{To find Tlambda f 1}
\end{align}
The last line follows from Lemma \ref{spherical property of b lambda}. Next, \begin{align}
\what{f}_H (\lambda) b_\lambda(a_t)&= b_\lambda(a_t) \int_K \int_0^\infty \int_K f(k_1 a_s k_2) \phi_{\sigma,\lambda}^{n,n}((k_2^{-1} a_{-s} k_1^{-1}) \Delta(s) dk_1 ds dk_2 \notag \\
 &= b_\lambda(a_t) \int_0^\infty f(a_s) \phi_{\sigma,\lambda}^{n,n}(a_s) \Delta(s) ds \hspace{3mm}(\textup{Since $\phi_{\sigma,\lambda}^{n,n}(a_{-s})= \phi_{\sigma,\lambda}^{n,n}(a_s)$}) \,. \label{To find Tlambda f 2}
\end{align}
Putting the expressions above \eqref{To find Tlambda f 1} and \eqref{To find Tlambda f 2} in the definition of $T_\lambda f$ the result follows.\\
\end{proof}
Next we show $T_\lambda f$ is in $L^1(G)_{n,n}$ for $0 < \Re \lambda <1$ and find the estimates of $||T_\lambda f||_1$.
\begin{lemma}\label{L-1 estimates T-lambda-f}
Let $0<\Re\lambda<1$ and $f$ be a $(n,n)$ type integrable function on $G$. Then $T_\lambda f\in L^1(G)_{n,n}$ and moreover if $\lambda\notin B(0;1) \cup B(1;1)$, its $L^1$ norm satisfies, $$||T_\lambda f||_1\leq C||f||_1(1+|\lambda|) d(\lambda,\partial S_1)^{-1},$$  where $d(\lambda,\partial S_1)$ denotes the Euclidean distance of $\lambda$
from the boundary $\partial S_1$ of the strip $ S_1$. 
\end{lemma}
\begin{proof} Proof of the Lemma above follows exactly in the same line as \cite[Lemma 4.4]{PS}.
\end{proof}
Now we summarize the necessary properties of the resolvent transform.
\begin{lemma} \label{properties of resolvent transform}
Assume $g\in L^\infty(G)_{n,n}$ annihilates $I$ and fix a function $f\in I$. Let $Z(\widehat{f}_H):=\{z\in S_1:\widehat{f}_H(z)=0\}$. Then

\noindent (a) $\mathcal{R}[g](\lambda)$ is an even holomorphic function on $\C $. It is given by the following formula :
\begin{eqnarray*}
\mathcal{R}[g](\lambda)=
\begin{cases} 
\langle b_\lambda,g\rangle,\hspace{3mm}\Re \lambda>1, \lambda \not\in \textbf{B}\\
\frac{\langle T_\lambda f,g\rangle}{\widehat{f}_H(\lambda)}, \hspace{3mm}0<\Re\lambda<1, \lambda\notin Z(\widehat{f}_H).
\end{cases}
\end{eqnarray*}
\noindent (b) For $|\Re\lambda|>1$, $\left|\mathcal{R}[g](\lambda)\right|\leq C||g||_\infty\frac{(1+|\lambda|)}{d(\lambda,\partial S_1)},$

\noindent (c) For $|\Re\lambda|<1$, $\left|\widehat{f}_H(\lambda)\mathcal{R}[g](\lambda)\right|\leq C||f||_1||g||_\infty\frac{(1+|\lambda|)}{d(\lambda,\partial S_1)}$, where the constant $C$ is independent of $f\in I$. 
\end{lemma}

\begin{proof}
(a) \textbf{CASE-1 } : Let $\Re \lambda >1 $ and $\lambda \not\in \textbf{B}$ then by (\ref{when b lambda is in L1}) $b_\lambda$ is in $L^1(G)_{n,n}$. % and from (\ref{fourier transform of b lambda in S1})   we have $\widehat{b}_{\lambda H}(z)=\frac{1}{\lambda^2-z^2}, \, z \in S_1$, \,  $\widehat{b}_{\lambda B}(k)=\frac{1}{\lambda^2-k^2}, \, k \in \Gamma_n$. Hence we have. 
For $z \in S_1$ we have from Lemma \ref{fourier transform of b lambda} and \ref{fourier transform of b lambda in S1},
\begin{align*}
&\frac{1}{\widehat{b}_{\lambda_0 H}(z)}-\frac{1}{\widehat{b}_{\lambda H}(z)}=\lambda_0^2 -\lambda^2\\
&\hspace{-.4 in} \textup{so}, \quad \left(1-(\lambda_0^2 -\lambda^2)
\widehat{b}_{\lambda_0 H}(z)\right)\widehat{b}_{\lambda H}(z)=\widehat{b}_{\lambda_0 H}(z).
\end{align*} 
Similarly for $k \in \Gamma_n$ we have,
$$\left(1-(\lambda_0^2 -\lambda^2)
\widehat{b}_{\lambda_0 B}(k)\right)\widehat{b}_{\lambda B}(k)=\widehat{b}_{\lambda_0 B}(k).$$
So $$\left( \delta -({\lambda_0}^2- \lambda^2) b_{\lambda_0}(\cdot)  \right) b_\lambda(\cdot) =b_{\lambda_0}(\cdot)$$ as $L^1_\delta(G)_{n,n}$ functions. Hence in the quotient algebra $L^1_\delta(G)_{n,n} /I$,
\begin{equation}\label{3.3}
\left(\delta-(\lambda_0^2 - \lambda^2)
b_{\lambda_0}+I\right)*(b_\lambda+I) =b_{\lambda_0}+I,
\end{equation}
Now $(\delta-(\lambda_0^2 - \lambda^2)b_{\lambda_0}+I )$ is invertible in  $L^1_\delta(G)_{n,n} /I$ so from
 (\ref{definition of B-lambda}) and (\ref{3.3})we get $B_\lambda = b_\lambda + I$ Therefore by the definition of $\mathcal{R}[g](\lambda)$, $$\mathcal{R}[g](\lambda) =\left\langle b_\lambda,g\right \rangle.$$
 
 \textbf{CASE-2 } : Let  $0<\Re\lambda<1$, $\lambda\notin Z(\widehat{f}_H)$. Then by Lemma \ref{L-1 estimates T-lambda-f} $T_\lambda f$ is in $L^1(G)_{n,n}$. % and $ \widehat{T_\lambda f}_H(z)=\frac{\widehat{f}_H(\lambda)-\widehat{f}_H(z)}{\lambda^2  - z^2} \hspace{3mm}\textup{for all} \, z \in S_1$, \, $ \widehat{T_\lambda f}_B(k)=\frac{\widehat{f}_H(\lambda)-\widehat{f}_B(k}{\lambda^2  - k^2} \hspace{3mm}\textup{for all} \, k \in \Gamma_n.$
  Similarly  as in previous case we have from Lemma \ref{Fourier transform of T lambda f},
\begin{align*}
 \quad \left(1-(\lambda_0^2 -\lambda^2)
\widehat{b}_{\lambda_0 H}(z)\right) \dfrac{\what{T_\lambda f}_H(z)}{\what{f}_H (\lambda)} = \widehat{b}_{\lambda_0 H}(z) - \dfrac{\what{f}_H(z) \what{b_{\lambda_0 }}_H(z)}{\what{f}_H(\lambda)} \hspace{3mm} \textup{ for all $z \in S_1$}
\end{align*}
 and
 \begin{align*}
\quad \left(1-(\lambda_0^2 -\lambda^2)
\widehat{b}_{\lambda_0 B}(k)\right) \dfrac{\what{T_\lambda f}_B(k)}{\what{f}_H (\lambda)} = \widehat{b}_{\lambda_0 B}(k) - \dfrac{\what{f}_B(k) \what{b_{\lambda_0 }}_B(k)}{\what{f}_H(\lambda)} \hspace{3mm} \textup{ for all $k \in  \Gamma_n.$} 
\end{align*}
Therefore $$
\left(\delta-({\lambda_0}^2-\lambda^2)
b_{\lambda_0}(\cdot) \right)\left(\frac{T_\lambda f(\cdot)}{\widehat{f}_H(\lambda)} \right) =b_{\lambda_0}(\cdot)- \dfrac{f(\cdot) b_{\lambda_0}(\cdot)}{\what{f}_H(\lambda}
$$ in $L^1_\delta(G)_{n,n}$. Since $f \in I $ %then in the quotient algebra $L^1_\delta(G)_{n,n} /I$,
\begin{equation}
\left(\delta-(\lambda_0^2-\lambda^2)
b_{\lambda_0}+I\right)*\left(\frac{T_\lambda f}{\widehat{f}(\lambda)}+I\right) =b_{\lambda_0}+I.
\end{equation}
Again from \eqref{definition of B-lambda} and the equation above $$B_\lambda=\frac{T_\lambda f}{\widehat{f}(\lambda)}+I,$$ which implies  $$\mathcal{R}[g](\lambda) =\dfrac{\left\langle {T_\lambda f}, g \right\rangle}{\widehat{f}(\lambda)}.$$

(b) Since $\mathcal{R}[g](\lambda)$ is even we only need to consider the case $\Re \lambda >1$. For $\Re \lambda >1 $  and  $\lambda \not\in \textbf{B}_1$ we have from Lemma \ref{L-1 norm estimates of b-lambda},
$$||b_\lambda||_1\leq C\frac{(1+|\lambda|)}{d(\lambda,\partial S_1)}
\hspace{3mm} \textup{for some $C>0$}.$$
Now from (\ref{defn-R-g}) it follows that $\mathcal{R}[g](\lambda)$ is bounded on $\textbf{B}_1$. Hence $$\left|\mathcal{R}[g](\lambda)\right|\leq C||g||_\infty\frac{(1+|\lambda|)}{d(\lambda,\partial S_1)}.$$\\
(c)  From Lemma \ref{L-1 estimates T-lambda-f} we get for $0 < \Re \lambda <1 $ and $\lambda \not\in B(0 ;1 ) \cup B(1 ; 1)$,
$$\left|\widehat{f}_H(\lambda)\mathcal{R}[g](\lambda)\right|\leq C||f||_1||g||_\infty\frac{(1+|\lambda|)}{d(\lambda,\partial S_1)}.$$ 
 Since $\widehat{f}_H(\lambda)\mathcal{R}[g](\lambda)$ is an even continuous function on $ S_1$, the same estimate is true for $0<|\Re\lambda|<1, \lambda\not\in   B(0;1) \cup B(1; 1 )$. Now from (\ref{defn-R-g}) it follows that $\mathcal{R}[g](\lambda)$ is bounded on $  B(0 ; 1 ) \cup B(1 ;1) $ with bound independent of $f$, Therefore for $0<|\Re\lambda|<1$ and $ \lambda \in B(0 ;1 ) \cup B(1 ; 1) $ , $$\left|\widehat{f}(\lambda)\mathcal{R}[g](\lambda)\right|\leq C||f||_1 ,$$ where $C$ is independent of $f$ and $\lambda$. So we have  for $0<|\Re\lambda|<1$, 
$$\left|\widehat{f}_H(\lambda)\mathcal{R}[g](\lambda)\right|\leq C||f||_1||g||_\infty\frac{(1+|\lambda|)}{d(\lambda,\partial S_1)}.$$
Finally the constant in the inequality above is independent of $f$ so by continuity of $\mathcal{R}[g]$ and $\what{f}$ the lemma follows.
\end{proof}
\section{\textbf{Results from complex analysis}} 
For any function $F$ on $i\R$, we let $$\delta_\infty^+(F)=-\limsup_{t\rightarrow\infty} e^{-\frac{\pi}{2}t}\log|F(it)| \,  \,  \textup{ and} \,  \, \delta_\infty^-(F)=-\limsup_{t\rightarrow\infty} e^{-\frac{\pi}{2 }t}\log|F(-it)|.$$
Next from \cite[Theorem 6.3]{PS} we have the following theorem.
 \begin{theorem}\label{Complex analysis Theorem}
 Let $M:(0,\infty)\rightarrow (e,\infty)$ be a continuously differentiable decreasing function with 
$$
\lim_{t\rightarrow 0^+}t\log\log M(t)<\infty,\hspace{3mm}\int_0^\infty\log\log M(t)dt<\infty.
$$
Let $\Omega$ be a collection of bounded holomorphic functions on $ S_1^0$ such that 
$$
\inf_{F\in \Omega}\delta^+_\infty(F)=\inf_{F\in \Omega}\delta^-_\infty(F)=0.
$$ 
Suppose $H$ satisfies the following estimates for some nonnegative integer N:
\begin{eqnarray*}
|H(z)|&\leq &(1+|z|)^NM\left(d(z,\partial S_1)\right),
\hspace{3mm}z\in\mathbb{C}\setminus S_1,\\
|F(z)H(z)|&\leq & (1+|z|)^NM\left(d(z,\partial S_1)\right),
\hspace{3mm}z\in S_1^0,
\hspace{1mm}\textup{for all}\hspace{1mm}F\in \Omega.
\end{eqnarray*}
\begin{enumerate}
\item[1.]If in addition, H is a holomorphic function on $S_1 \setm \{ \pm 1 \}$ then H is dominated by a polynomial outside a bounded neighbourhood of $\{ \pm 1\}.$
\item[2.] If H is an entire function, then it is a polynomial.
\end{enumerate}
\end{theorem}

\section{\textbf{Proof of W-T Theorem for $L^1(G)_{n,n}$}}
\begin{proof}[Proof of Theorem \ref{WTT for L1(G)n,n}]
Since the ideal generated by $\{f^\alpha\mid \alpha\in\Lambda\}$ is same as the ideal generated by the elements $\left\{\frac{f^\alpha}{||f^\alpha||}\mid \alpha\in\Lambda\right\}$ and $\delta^{\pm}_\infty(\widehat{f}_H)=
\delta^{\pm}_\infty
\left(   \dfrac{\what{f^\alpha_H}}{||f||_1}\right)$, we can assume that the functions $f^\alpha$ are of unit $L^1$ norm. Let $g\in L^\infty(G)_{n,n}$ annihilates the closed ideal $I$ generated by $\{f^\alpha\mid \alpha\in\Lambda\}$. We will show that $g = 0 $.  Then by an application of Hahn Banach theorem it will follow that $I = L^1(G)_{n,n}$.
From the hypothesis we have,
$$
\inf_{\alpha\in\Lambda}\delta^+_\infty(\widehat{f^\alpha_H})=\inf_{\alpha\in\Lambda}\delta^-_\infty(\widehat{f^\alpha_H})=0.
$$
By Lemma \ref{properties of resolvent transform}, the entire function $\mathcal{R}[g]$ satisfies the following estimates
\begin{eqnarray*}
|\mathcal{R}[g](z)|&\leq & C(1+|z|)\left(d(z,\partial S_1)\right)^{-1},
\hspace{3mm}z\in\mathbb{C}\setminus S_1,\\
|\widehat{f^\alpha_H}(z)\mathcal{R}[g](z)|&\leq & C(1+|z|)\left(d(z,\partial S_1)\right)^{-1},
\hspace{3mm}z\in S_1^0,
\end{eqnarray*}
for all $\alpha\in\Lambda$, where $C$ is a constant and we choose it is  greater than $e$. We can define $M:(0,\infty)\rightarrow(e,\infty)$ to be a continuously differentiable decreasing function such that
$M(t)=\frac{C}{t}$ for $0<t< 1$, and  $\int_1^\infty\log\log M(t)dt<\infty$. With this definition of $M$, we  have 
\begin{eqnarray*}
|\mathcal{R}[g](z)|&\leq & (1+|z|)M\left(d(z,\partial S_1)\right)
\hspace{3mm}z\in\mathbb{C}\setminus S_1,\\
|\widehat{f^\alpha_H}(z)\mathcal{R}[g](z)|&\leq & (1+|z|)M\left(d(z,\partial S_1)\right)
\hspace{3mm}z\in S_1^0, \hspace{2mm} \textup{for all $\alpha \in \Lambda$}.
\end{eqnarray*}
Therefore, by Theorem \ref{Complex analysis Theorem}, $\mathcal{R}[g](z)$ is a polynomial. From Lemma \ref{properties of resolvent transform}, $$\mathcal{R}[g](z) \leq ||b_z||_1 ||g||_\infty.$$   Then  Lemma \ref{L-1 norm estimates of b-lambda} implies  $\mathcal{R}[g](z)\rightarrow 0$ when $z \rightarrow \infty$ along the positive real axis. Therefore $\mathcal{R}[g]$  must be the zero polynomial. Hence  $\langle b_\lambda,g\rangle =0$  whenever $\Re \lambda >1$  and $\lambda \not\in \textbf{B}$ but the collection $\{ b_\lambda | \, \Re \lambda >1 \text{ and } \lambda \not\in \textbf{B} \}$ spans a dense subset of $L^1(G)_{n,n}$ by Lemma \ref{b lambda dense in L1(G)n,n }. So $g= 0$ and the proof follows.
\end{proof}

 Finally we like to mention here that we first started to prove a W-T theorem for $L^1(G)_{m,n}$ but our method fails in this general setting as $L^1(G)_{m,n}$ is not necessarily a commutative banach algebra.
\section{\textbf{Final Results}  } 
Now we prove Wiener Tauberian theorem for $L^1(G)_n$ using  Theorem \ref{WTT for L1(G)n,n}. Here we will follow similar technique as in \cite{Sarkar-1997}. %Idea to prove the Theorem \ref{WTT for L1(G)n} is that from the given collection $\lbrace f^\alpha \mid \alpha \in \Lambda \rbrace$ we will construct another collection which will satisfy the hypothesis of Theorem \ref{WTT for L1(G)n,n} for all $n \in \Z$ and then from a general fact our theorem will follow. This is the  same technique used in \cite{Sarkar-1997}.\\
 
 For $f \in L^1(G)$ we have from \cite[p. 30, prop 7.3]{Barker},
 \begin{equation}
(\what{f}_B(k) )_{m,n} =\eta^{m,n}(k) \,  (\what{f}_H(k)  )_{m,n} \text{ for all $k \in \lbrace \pm 1\rbrace $ and $m,n  \in \Z (k),$}
  \end{equation}
   where $\eta^{m,n}(k)$ is a positive number. Therefore $$(\what{f}_B(k) )_{m,n} \neq 0 \Leftrightarrow (\what{f}_H(k) )_{m,n} \neq 0.  $$
   Suppose  $\what{f}_B(k) \neq 0$ for all $k \in \Gamma_n$, then it implies the following:
   \begin{enumerate}
   \item[(a)] If $n$ is positive then for every $ m < n , \,  (\what{f}_B(n-1))_{m,n} =0$, so $f $ has at least one non zero component of left type $m$ such that $m \geq n$. Similarly when $n$ is negative $f$ has at least one left type $m$ for some $m \leq n$.
   \item[(b)] Let $f \in L^1(G)_n$ and $n$ is even. If $n > 0$ then by the hypothesis above $\what{f}_B(1) \neq 0$ and so there is an $m$ such that $m \in \Z(1)$ and $(\what{f}_B(1) )_{m,n} \neq 0$. Therefore $(\what{f}_H(1) )_{m,n} \neq  0$ . For $n < 0 $ one can have a similar statement.
 \end{enumerate}
 \begin{proof}[Proof of Theorem \ref{WTT for L1(G)n}.]
 We first consider the case when the collection indexed by $\Omega$ contains exactly one function,  $f \in L^1(G)_n$. Let $ f_m(x) = \int_{0}^{2 \pi} e^{-im\theta} f(k_\theta x)  \, d \theta$ for all $m \in \Z$. Then $f_m$ is an $(m,n)$ type function and $(m,n)$-th matrix coefficient $\what{f}_H$, $(\what{f}_H  )_{m,n} = \what{f_m}_H$.
 
\indent Now we will construct a family of functions in $\C$, $\lbrace  \mathcal{G}_m (\cdot)  \mid m\in \Z^\sigma \rbrace$  such that $\mathcal{G} \in  \mathcal{C}^1_H(\what{G} )_{n,m} $. 
%where $\mathcal{G}_m(\lambda) =e^{-\lambda^4} Q_{n,m}(\lambda)$ and $Q_{n,m}(\lambda)$ are polynomials  will be  depending on $m,n$. 

 %Let $\mathcal{G}_m(\lambda) =e^{-\lambda^4} Q_{n,m}(\lambda) $ where  $Q_{n,m}$ is to be a chosen polynomial depending on $m$ and $n$. 

\noindent When $m n \geq 0$ let us define $\mathcal{G}_m(\lambda) =e^{-\lambda^4} Q_{n,m}(\lambda) $ where  $Q_{n,m} = P_{n,m}$ which is the numerator of the rational function $\varphi^{n,m}_\lambda$ from \eqref{peref definition of Pmn lambda}. Hence $e^{-\lambda^4} Q_{n,m}(\lambda) = \varphi^{n,m}_\lambda e^{-\lambda^4} Q_{n,m}(-\lambda) $ which shows that 
\begin{equation}
 \mathcal{G}_m(\lambda) =e^{-\lambda^4} Q_{n,m}(\lambda)  \in  \mathcal{C}^1_H(\what{G} )_{n,m}\end{equation}
 for the case $m   n \geq 0$. Here we note that $Q_{n,m}(0) \neq 0$.
 %%%%%%%%%%%%%%%%%%%%%%%%%%%%%%%%%%%%%%%%%%%%%%%%%5%%%%%%%%%%%%%%%%%%%%%%%%%%55%%%%%Here I need to mention which property later%
 
\noindent If $m  n <0 $ then we will have to choose the polynomial in a slightly different way because  we want $ \mathcal{G}_m(\lambda)$ to satisfy all the properties of $\mathcal{C}^1_H(\what{G} )_{n,m}$.
\begin{enumerate}
\item[\textsf{ Case 1.}] Let $n$ be odd. Then we take the polynomial $Q_{n,m}' (\lambda) =P_{n,m}(\lambda) \cdot \lambda^2.$ Now $Q_{n,m}' (0) = 0$ and $e^{-\lambda^4} Q_{n,m}'(\lambda) = \varphi^{n,m}_\lambda e^{-\lambda^4} Q_{n,m}'(-\lambda) $. Therefore, in this case 
\begin{equation}
 \mathcal{G}_m(\lambda) =e^{-\lambda^4} Q_{n,m}'(\lambda)  \in  \mathcal{C}^1_H(\what{G} )_{n,m}.
 \end{equation}
\item[\textsf{ Case 2.}] Let $n$ be even( hence $|n| , |m| \geq 2$ as $n  m < 0$ ). Then the required polynomial is $Q_{n,m}'' (\lambda) =P_{n,m}(\lambda) (1-\lambda^2) $. So $Q_{n,m}'' ( \pm 1) = 0$ and $e^{-\lambda^4} Q_{n,m}''(\lambda) = \varphi^{n,m}_\lambda e^{-\lambda^4} Q_{n,m}''(-\lambda) $. Therefore in this case also, 
\begin{equation}
 \mathcal{G}_m(\lambda) =e^{-\lambda^4} Q_{n,m}''(\lambda)  \in  \mathcal{C}^1_H(\what{G} )_{n,m}.
 \end{equation}
\end{enumerate}
Now for all $n,m $
\begin{align*}
\mathcal{G}_m(\lambda)  \what{f_m}_H (\lambda)&= e^{-\lambda^4}  Q_{n,m}(\lambda) \what{f_m}_H(\lambda)\\
& = e^{-\lambda^4}  Q_{n,m}( - \lambda) \, \varphi^{n,m}_\lambda \, \varphi^{m,n}_\lambda \what{f_m}_H(-\lambda)\\
&= \mathcal{G}_m(- \lambda)  \what{f_m}_H (- \lambda).
\end{align*}
Since $f_m $ is an $(m,n)$ type function on $G$ so $\what{f_m}_H ( \lambda) = \varphi^{m,n}_\lambda \what{f_m}_H ( -\lambda)$ and $ \varphi^{n,m}_\lambda = (\varphi^{m,n}_\lambda)^{-1}$. This shows that for  all $m$,  $\mathcal{G}_m(\lambda)  \what{f_m}_H (\lambda)$ is the Fourier transform of an $(n,n)$ type function with respect to principal series representation. Now we claim that $\lambda\in S_1$ there is an $m$ such that $\mathcal{G}_m(\lambda)  \what{f_m}_H (\lambda) \neq 0$. The only possible zeros of the polynomials $ Q_{n,m} ,  Q_{n,m}' $ and $ Q_{n,m}'' $ in $S_1$ are $\lbrace 0, \pm 1 \rbrace$ and everywhere else it is non-zero. Given $\what{f}_H(\lambda) \neq 0$ for all $\lambda \in S_1$. If we could show that for each $\lambda \in \lbrace 0, \pm 1 \rbrace $ there is an $m$ such that $\mathcal{G}_m(\lambda)  \what{f_m}_H (\lambda) \neq 0$ then we will be done.

 Before proving our claim we find out exactly when $\lbrace 0, \pm 1 \rbrace $ are zeros of the polynomials above.
\begin{enumerate}
\item[(i)] $ P_{n,m}(-1) =0 $ if and only if $n =0$ and $ m \neq 0$,\\
 $ P_{n,m}(+1) \neq 0 $ for all $m \neq 0$ and $ P_{n,0}(+1) \neq 0$ when $n \neq 0$,\\
 therefore \begin{equation}\label{zeros of Q n,m case 1}
  Q_{n,m}(\pm1) \neq 0 \text{ when } nm \neq 0 \text{ and } Q_{n,m}'(\pm1) \neq 0 \text{ for all }m \neq 0 .
 \end{equation}
 \item[(ii)] \begin{eqnarray} 
   \text{ Since }Q_{n,m}(0) \neq 0 \text{ so } Q_{n,m}''(0) \neq 0.\label{zeros of Q n,m case 2} \\ Q_{n,m}'(0) = 0 \text{ and } Q_{n,m}''(\pm 1) = 0.\label{zeros of Q n,m case 3}
   \end{eqnarray}
\end{enumerate}

First we consider the case for $\lambda =0.$ By hypothesis there is an $m$ such that $\what{f_m}_H(0) \neq 0 $. 

\noindent If $n$ is odd then $m  n > 0$,  otherwise  $\phi^{m,n}_{\sigma^- , 0} \equiv 0$ which implies $\what{f_m}_H(0) = 0 $. Therefore 
 $\mathcal{G}_m(0)  \what{f_m}_H (0) \neq 0$ as $Q_{n,m}(0) \neq 0$.
 
\noindent Next suppose $n$ is even. Now If $n m \geq 0$ then  $\mathcal{G}_m(0)  \what{f_m}_H (0) \neq 0$ as $Q_{n,m}(0) \neq 0$. When $nm < 0 $ then also $\mathcal{G}_m(0)  \what{f_m}_H (0) \neq 0$ because from \eqref{zeros of Q n,m case 2} $Q_{n,m}''(0) \neq 0$ . 

Now we prove our claim for $\lambda = \pm 1 $. Here we will consider several case for $n$. \begin{enumerate}
\item[\textsf{Case 1.}] Let $n=0$, then $\what{f_m}_H (1) = 0 $ for all $m \neq 0$ as $\phi^{m,0}_{\sigma^+ , 1} \equiv 0$. Therefore $ \what{f_0}_H (1) \neq 0 $ and also $ \what{f_0}_H (1) =  \what{f_0}_H (-1) .$
\item[\textsf{Case 2.}] Let $n (\neq 0)$ be an even no. If $n > 0$ then by discussion(b) preceding this proof, there exists an $r \in \Z(1)$ such that $ \what{f_r}_H (1) \neq 0 $ and so $\mathcal{G}_r(1)  \what{f_r}_H (1) \neq 0$ (since $Q_{n,m}(\pm 1) \neq 0$ for $nm >0$ see \eqref{zeros of Q n,m case 1}). But $\what{f_r}_H ( -1) = \varphi^{n,r}_1 \what{f_r}_H ( 1)$ and $\varphi^{n,r}_\lambda$ has no zero at $\lambda =1$ (see \cite[prop 7.2]{Barker}. This shows that $\what{f_r}_H ( -1) \neq 0$ and so  $\mathcal{G}_m(-1)  \what{f_m}_H (-1) \neq 0$.

\noindent When $n < 0$ we will give similar arguments. By the same discussion(b) there exists an $s \in \Z(-1)$ such that $ \what{f_s}_H (-1) \neq 0 $ and so $\mathcal{G}_s(-1)  \what{f_s}_H (-1) \neq 0$. But $\what{f_s}_H ( -1) = \varphi^{n,s}_1 \what{f_s}_H ( 1)$ and $\varphi^{n,s}_\lambda$ has no pole at $\lambda =1$. This implies $ \what{f_s}_H (1) \neq 0 $ hence $\mathcal{G}_m(-1)  \what{f_m}_H (-1) \neq 0$ (since $Q_{n,m}(\pm 1) \neq 0$ for $nm >0$ see \eqref{zeros of Q n,m case 1}). This concludes our claim when $n$ is an even no.
\item[\textsf{Case 3.}] Let $n$ be an odd no. Then by the hypothesis there exist $ m \in \Z^{\sigma^-}$ such that $ \what{f_m}_H (1) \neq 0 $. Then  $mn \neq 0$ so from \eqref{zeros of Q n,m case 1} it follows $\mathcal{G}_m(1)  \what{f_m}_H (1) \neq 0$. Proof for $\lambda = -1$ is exactly similar.

\end{enumerate}

Let $\mathcal{G}_m'(k) =e^{-k^4} Q_{n,m}(k) $ for all $k \in \Gamma_n$ where $Q_{n,m}$ is chosen in the same way as before . Now  let for $k_0 \in \Gamma_n$, $ \what{f_{m_0}}_B (k_0) \neq 0$ then $m_0 \in \Z(k_0).$ Therefore $Q_{n,m_0}(k_0) \neq 0 $ as all the zeros of the polynomial $P_{n,m_0}$ are either between $m_0$ and $n$ or between $-m_0$ and $-n$ (see \cite[prop. 7.1]{Barker}). Now from Lemma \ref{Isomorphism of schwarz space} isomorphism  between $\mathcal{C}^1(G)_{n,m}$ and $\mathcal{C}^1(\what G)_{n,m}$ for every $m$, there exists $g_m \in \mathcal{C}^1(G)_{n,m}$ such that $\what{g_m}_H (\lambda) =\mathcal{G}_m (\lambda)$ for all $\lambda \in S_1$ and $\what{g_m}_B (k) =\mathcal{G}_m (k)$ for  all $k \in \Gamma_n$.

Now we show the set of $L^1(G)_{n,n}$ functions $\lbrace  g_m * f_m \mid m \in \Z^\sigma \rbrace$  satisfies all the conditions of Theorem \ref{WTT for L1(G)n,n}. Since $Q_{n,m} $'s are always polynomial in $\lambda$, by a simple argument of analysis shows that
%there is a large $t_0$ such that for all $t \geq t_0 $,  
%\noindent $ \mid e^{-t^4} Q_{n,m}(it) \mid \leq \frac{C}{t}. $ Since $\log$ is an increasing function so we get for all $m$, 
\begin{equation}
\lim\limits_{t\rightarrow \infty} e^{-\frac{\pi}{2}t} \log \mid \mathcal{G}_m(it)\mid = 0 \label{limsup of gm is zero}.
\end{equation}
Hence,
\begin{align}
\limsup\limits_{t\rightarrow \infty} e^{-\frac{\pi}{2}t} \log \mid \mathcal{G}_m(it) \what{f_m}_H(it)\mid &= \lim\limits_{t\rightarrow \infty} e^{-\frac{\pi}{2}t} \log \mid (\mathcal{G}_m(it))\mid + \limsup\limits_{t\rightarrow \infty} e^{-\frac{\pi}{2}t} \log \mid \what{f_m}_H(it)\mid \notag \\
& =\limsup\limits_{t\rightarrow \infty} e^{-\frac{\pi}{2}t} \log \mid \what{f_m}_H(it)\mid  \label{limsup of the new collection is same as old one}. 
\end{align}
Therefore by the given hypothesis, \begin{equation}
\inf\limits_{m \in \Z^\sigma} \delta^{\pm }(\mathcal{G}_m  \what{f_m}_H) = 0.
\end{equation}
So we have established that the ideal generated by  $\lbrace  g_m * f_m \mid m \in \Z^\sigma \rbrace$ is dense in $L^1(G)_{n,n}$. But $g_m* f_m = g_m * f$; so the result follows from the fact that the left $L^1(G)$ module generated by $L^1(G)_{n,n}$ is all of $L^1(G)_n.$

Now suppose  $\Lambda$ is an arbitrary index set.  Then out of each $f^\alpha$ by projections we get $f^\alpha_j$ for all $j \in \Z$ which are functions of type $(j,n)$. We apply previous arguments to the collection $\lbrace \what{f^\alpha_{jH}} \mid \alpha \in \Lambda,\, j \in \Z \rbrace$ of functions in $L^1(\what G )_n$ and the theorem follows.\\
 \end{proof}

\begin{proof}[Proof of Theorem \ref{WTT for L1(G)}]
As we have seen in the proof of previous theorem , it is enough to consider the case when the collection contains a single function, namely  $f$. Let $f_j$ be the projection of $f$ to $L^1(G)_j$, for every $ j \in \Z$. For each $j,m \in \Z$, we choose a polynomial $Q_{j,m}$ in $\lambda$ involving $j$ and $m$ so that  $ e^{-\lambda^4} Q_{j,m}(\lambda)  \in \mathcal{C}^1_H(\what G)_{j,m}$.

 When $jm  \geq 0$, $Q_{j,m} = P_{j,m}$  is the numerator of the rational function $\varphi^{j,m}_\lambda$. Now suppose $jm < 0$, then whenever  $j,m$ are odd integers  we take $Q_{j,m}^{'} = \lambda^2 P_{j,m}$ and  if $j,m$ are even integers then we choose  $Q_{j,m}^{''} =(1- \lambda^2) P_{j,m}$, where $P_{j,m}$ is as above.  Then for $m\in \Z$,   $ e^{-\lambda^4} Q_{j,m}(\lambda)  \in \mathcal{C}^1_H(\what G)_{j,m}$. By the isomorphism of $L^1$ Schwartz space $\mathcal{C}^1(G)_{j,m}$ and  $\mathcal{C}^1(\what G)_{j,m}$ (see Lemma \ref{Isomorphism of schwarz space}) there exists $g_{j,m} \in \mathcal{C}^1(G)_{j,m}$ such that $\what{g_{j,m}}_H(\lambda) = e^{-\lambda^4} Q_{j,m}(\lambda)   $ for all $\lambda \in S_1$ and $\what{g_{j,m}}_B(k) = e^{-k^4} Q_{j,m}(k) $ for all $k \in \Gamma_j$. Now for all $m \in \Z$ we consider the following collection of functions,
\begin{equation*}
\mathcal{F}_m = \lbrace f_j * g_{j,m} \mid j \in \Z \rbrace
\end{equation*}
contained in $L^1(G)_m$.

As in \eqref{limsup of gm is zero} and \eqref{limsup of the new collection is same as old one} we have for each $m \in \Z$,
\begin{equation}
\limsup\limits_{t\rightarrow \infty} e^{-\frac{\pi}{2}t} \log \mid \what{g_{j,m}}_H(it) \what{f_{i,j}}_H(it)\mid = \limsup\limits_{t\rightarrow \infty} e^{-\frac{\pi}{2}t} \log \mid \what{f_{i,j}}_H(it)\mid 
\end{equation}
for all $i,j \in \Z$. So,
\begin{equation}
\inf\limits_{i,j \in \Z} \delta^{\pm }(\what{g_{j,m}}_H \what{f_{i,j}}_H) = 0.\label{delta infimum of gim fij is same as gim}
\end{equation}
Now for all $m \in \Z $, Fourier transforms of the elements of $\mathcal{F}_m$ does not have common zeros, follows from \cite[Theorem 1.2]{Sarkar-1997}. Therefore together with \eqref{delta infimum of gim fij is same as gim} it follows that for every $m$, elements of $\mathcal{F}_m$ satisfies all the conditions of Theorem \ref{WTT for L1(G)n} and so $\mathcal{F}_m$ generates $L^1(G)_m$ under left convolution. Now $ f_j * g_{j,m} =  f * g_{j,m}$, for every m. So the two sided closed ideal generated by $f$  contains $L^1(G)_m$ for all $m$.  The smallest closed right $G$-invariant subspace of $L^1(G)$ containing $L^1(G)_m$ for all $m \in \Z$, is $L^1(G)$ itself. Hence the first part of the Theorem \ref{WTT for L1(G)} follows.  The second part of the theorem follows similarly as in \cite[Theorem 1.2]{Sarkar-1997}.

\end{proof}

\noindent \textbf{Acknowledgement.} The author would like to   thank to his supervisor Prof. Sanjoy Pusti for introducing him to the problem and for the many useful discussions during the course of this work. I am grateful to him for encouraging me in research and his guidance.

%where $\Omega = (H^2+H- \overline{Y}Y$) 
\end{document}